\documentclass[a4paper,11pt]{article}

\usepackage[leqno]{amsmath}
\usepackage{amssymb,amsthm,MnSymbol}
\usepackage{mathrsfs}
\usepackage[all]{xy}
\usepackage[utf8]{inputenc}
\usepackage{enumitem}
\usepackage{hyperref}
\usepackage{tikz}
\usepackage{tikz-cd}
\usepackage{tocloft}
\usepackage{titlesec}

\titleformat*{\section}{\large\bfseries}

\theoremstyle{plain}
\newtheorem{thm}{Theorem}[section]
\newtheorem{prop}[thm]{Proposition}
\newtheorem{lem}[thm]{Lemma}
\newtheorem{cor}[thm]{Corollary}

\theoremstyle{definition}
\newtheorem{defin}[thm]{Definition}
\newtheorem{rem}[thm]{Remark}
\newtheorem{ex}[thm]{Example}

\newcommand{\A}{{\mathcal{A}}}
\newcommand{\ac}{{\operatorname{ac}}}
\newcommand{\C}{{\mathcal{C}}}

\newcommand{\cf}{\mathit{cf}}
\newcommand{\Ch}{{\operatorname{Ch}}}
\newcommand{\coker}{{\operatorname{coker}}}
\newcommand{\D}{{\mathcal{D}}}
\newcommand{\End}{{\operatorname{End}}}
\newcommand{\Ext}{{\operatorname{Ext}}}
\newcommand{\F}{{\mathcal{F}}}
\newcommand{\FF}{{\mathbb{F}}}
\newcommand{\GP}{{\operatorname{GP}}}
\newcommand{\Ho}{{\operatorname{Ho}}}
\newcommand{\Hom}{{\operatorname{Hom}}}
\newcommand{\id}{{\operatorname{id}}}
\newcommand{\im}{{\operatorname{im}}}
\newcommand{\ind}{{\operatorname{ind}}}
\newcommand{\Ind}{{\operatorname{Ind}}}
\newcommand{\Inj}{{\operatorname{Inj}}}
\newcommand{\Mod}{{\operatorname{Mod}}}
\newcommand{\op}{{\operatorname{op}}}
\newcommand{\Proj}{{\operatorname{Proj}}}
\newcommand{\Rep}{{\operatorname{Rep}}}

\newcommand{\T}{{\mathcal{T}}}

\newcommand{\XX}{{\mathbb{X}}}
\newcommand{\Z}{{\mathbb{Z}}}

\begin{document}
\setcounter{section}{-1}


{\Large\bf Model categories and pro-$p$ Iwahori-Hecke modules}\\[6ex]
{\sc Nicolas Dupr\'e, Jan Kohlhaase}\\[1ex]


{\bf Abstract.} Let $G$ denote a possibly discrete topological group admitting an open subgroup $I$ which is pro-$p$. If $H$ denotes the corresponding Hecke algebra over a field $k$ of characteristic $p$ then we study the adjunction between $H$-modules and $k$-linear smooth $G$-representations in terms of various model structures. If $H$ is a Gorenstein ring we single out a full subcategory of smooth $G$-representations which is equivalent to the category of all Gorenstein projective $H$-modules via the functor of $I$-invariants. This applies to groups of rational points of split connected reductive groups over finite and over non-archimedean local fields, thus generalizing a theorem of Cabanes. Moreover, we show that the Gorenstein projective model structure on the category of $H$-modules admits a right transfer. On the homotopy level the right derived functor of $I$-invariants then admits a right inverse and becomes an equivalence when restricted to a suitable subcategory.


\tableofcontents


\section{Introduction}\label{section_0}

%
%


Let $p$ denote a prime number and let $G$ denote a topological group admitting an open subgroup $I$ which is pro-$p$. If $k$ is a field of characteristic $p$ we let $\XX=\ind_I^G(k)$ denote the $G$-representation compactly induced from the trivial $I$-representation $k$. The opposite endomorphism ring $H=\End_G(\XX)^\op$ can also be realized as the double coset algebra $H=k[I\backslash G/I]$ with respect to the usual convolution product.\\

We let $\Rep_k^\infty(G)$ denote the category of $k$-linear $G$-representations which are {\it smooth} in the sense that the stabilizer of any vector is open in $G$. If $G$ is discrete then $\Rep^\infty_k(G)$ is the category of all $k$-linear $G$-representations. Let $\Mod(H)$ denote the category of left $H$-modules. There is an adjunction
\begin{equation}\label{FU_adjunction}
F:\Mod(H)\rightleftarrows\Rep^\infty_k(G):U
\end{equation}
given by $FM=\XX\otimes_HM$ and $UV=V^I\cong\Hom_G(\XX,V)$. The study of $G$-representations in terms of this adjunction plays a prominent role in various situations. If $G$ is the group of rational points of a split connected reductive group over a finite field of characteristic $p$, for example, and if $I$ denotes a $p$-Sylow subgroup of $G$ then this strategy has a long history.\\

Our motivating example, however, concerns the case that $G$ is the group of rational points of a split connected reductive group over a non-archimedean local field of residue characteristic $p$. If $I$ denotes a pro-$p$ Iwahori subgroup of $G$ in the sense of Bruhat-Tits theory then the algebra $H$ is known as the corresponding {\it pro-$p$ Iwahori-Hecke algebra}. In this situation the category $\Rep^\infty_k(G)$ plays a central role in the mod-$p$ local Langlands program. Although the category of $H$-modules is rather well-understood the properties of the category $\Rep^\infty_k(G)$ remain obscure. In fact, the behavior of the adjunction (\ref{FU_adjunction}) is not sufficiently clear beyond $G=\mathrm{GL}_2(\mathbb{Q}_p)$ and a few related cases (cf.\ the results of Ollivier in \cite{Oll09}).\\

Since $k$ is of characteristic $p$ the functor $U$ is generally not exact. It is therefore  natural to study the situation from a homological point of view. For reductive groups over non-archimedean local fields important work in this direction has been done by Ollivier, Schneider and Koziol (cf.\ \cite{OS14}, \cite{OS18} and \cite{Koz}, for example). Moreover, Schneider showed that the adjunction (\ref{FU_adjunction}) induces an equivalence on a derived level if $H$ is replaced by a suitable dg-variant $H_\bullet$ (cf.\ \cite{Sch}, Theorem 9). Unfortunately, the structure of this algebra seems very hard to understand although some recent progress has been made by Ollivier and Schneider (cf.\ \cite{OS19} and \cite{OS21}).\\

In this article we take a different approach. Our aim is to shed some light on the nature of the adjunction (\ref{FU_adjunction}) using the language and methods of {\it model categories}. We stay in the very general situation described at the beginning. For the finer results in \S\ref{section_4} and \S\ref{section_5} we will assume, however, that the ring $H$ is {\it Gorenstein}. In Example \ref{H_Gorenstein} we give a list of important situations to which this assumption applies. Due to results of Tinberg, Ollivier and Schneider this includes groups of rational points of split connected reductive groups over finite and non-archimedean local fields.\\

In \S\ref{section_1} we give a brief exposition of the general theory of model categories. Any model category $\C$ has an associated {\it homotopy category} $\Ho(\C)$ that may be thought of as a generalized derived category. The model structures re\-le\-vant for our purposes will either be constructed using {\it cotorsion pairs} in the sense of Hovey (cf.\ Theorem \ref{cotorsion_pairs}) or via the {\it left/right transfer} along an adjunction. We emphasize that cotorsion pairs differ from the {\it torsion theories} advertised in \cite{OS18}. In rather special situations, however, there is a correspondence between localizing cotorsion pairs and torsion theories in the homotopy category (cf.\ \cite{SS11}, Proposition 3.3). In Proposition \ref{path_objects} we give a criterion for the existence of the right transfer based on {\it path objects}.\\

Let $\Ch(H)$ and $\Ch(G)$ denote the categories of unbounded chain complexes over $\Mod(H)$ and $\Rep^\infty_k(G)$, respectively. By working termwise there is an induced adjunction $F:\Ch(H)\rightleftarrows\Ch(G):U$. In Proposition \ref{model_Gchain} and Proposition \ref{Quillen_equivalence} we show that the {\it projective model structure} on $\Ch(H)$ admits a right transfer to $\Ch(G)$ yielding a Quillen equivalence. In particular, the {\it derived adjunction}
\[
LF:D(H)\rightleftarrows\Ho(\Ch(G)):RU
\]
is an equivalence of categories. Here $D(H)$ denotes the usual unbounded derived category of the ring $H$. Although this result is somewhat formal the necessary input from relative homological algebra plays a major role in the sequel. In particular, this concerns the notion of an {\it $I$-exact sequence} of $G$-representations. This is an object $X\in\Ch(G)$ such that the complex $UX\in\Ch(H)$ of $I$-invariants is exact (cf.\ Lemma \ref{pullback}). Moreover, pulling back the notion of a projective $H$-module yields the relative notion of an {\it $I$-projective $G$-representation}. These are characterized in Lemma \ref{summand_I_proj}. It lies at the heart of many of our arguments that the adjunction (\ref{FU_adjunction}) restricts to an equivalence between the full subcategories of projective $H$-modules and $I$-projective $G$-representations, respectively (cf.\ Corollary \ref{proj_equiv}). \\

We point out that if $G$ is compact then the category $\Rep^\infty_k(G)$ does not have any non-zero projective objects unless the group $G$ is discrete (cf.\ Remark \ref{not_enough_projectives}). However, we show that the {\it injective model structure} on $\Ch(G)$ admits a left transfer to $\Ch(H)$ (cf.\ Proposition \ref{left_transfer_exists}). A detailed analysis of the corresponding Quillen adjunction will be given elsewhere.\\

In \S\ref{section_3} we collect known results about model structures over Gorenstein rings. We start with the notion of a {\it Gorenstein projective object} in an abelian category with enough projectives. If $S$ is a Gorenstein ring then Hovey constructed a model structure on $\Mod(S)$ for which the cofibrant objects are the Gorenstein projective $S$-modules and the trivial objects are the $S$-modules of finite projective dimension (cf.\ Theorem \ref{GP_R_Mod}). This is called the {\it Gorenstein projective model structure}. If $S=H$ comes from a split connected reductive group over a non-archimedean local field as in Example \ref{H_Gorenstein} (iv) then the corresponding homotopy category seems particularly well-suited to study the {\it supersingular} $H$-modules (cf.\ Remark \ref{homotopy_cat_trivial}). Returning to the general case there is a Quillen equivalent model structure on $\Ch(S)$ constructed by Becker. This is called the {\it singular projective model structure} (cf.\ Theorem \ref{Becker}). Its homotopy category is related to the singularity category of $S$ studied by Krause (cf.\ \cite{Kra05}). We emphasize that there is an extensive literature on homological algebra over Gorenstein rings starting with the seminal article \cite{Buc} of Buchweitz.\\ 

In \S\ref{section_4} we specialize to the case $S=H$ assuming that the ring $H$ is Gorenstein. In Proposition \ref{SP_right_transfer} we show that the singular projective model structure on $\Ch(H)$ admits a right transfer along the adjunction $F:\Ch(H)\rightleftarrows\Ch(G):U$ and yields a Quillen equivalence. On both sides the {\it loop} and {\it suspension functors} are induced by the usual shift functors (cf.\ Lemma \ref{explicit_loop_susp} and Corollary \ref{ISP_stable}). Moreover, the derived adjunction is an equivalence of categories
\[
LF:K_\ac(\mathrm{Proj}(H))\rightleftarrows\Ho(\Ch(G)):RU
\]
where the left hand side denotes the category of acyclic complexes of projective $H$-modules up to chain homotopy. The cofibrant objects of $\Ch(G)$ are the $I$-exact complexes of $I$-projective $G$-representations (cf.\ Lemma \ref{ISP_cofibrant}) and the cofibrations are made explicit in Corollary \ref{ISP_cofibrations}.\\

In Proposition \ref{GP_right_transfer} we construct suitable path objects to show that the Gorenstein projective model structure on $\Mod(H)$ admits a right transfer along the adjunction (\ref{FU_adjunction}). The analysis of the derived adjunction 
\begin{equation}\label{LFRU_adjunction}
LF:\Ho(\Mod(H))\rightleftarrows\Ho(\Rep^\infty_k(G)):RU
\end{equation}
is much less formal than in the previous situations because the cofibrant $H$-modules are only Gorenstein projective. Our main observation is that for a Gorenstein projective module the unit of the adjunction (\ref{FU_adjunction}) is a split monomorphism with an explicit cokernel (cf.\ Proposition \ref{GP_split_mono}). Moreover, we can make explicit the homotopy relation on morphisms between cofibrant objects (cf.\ Lemma \ref{IGP_homotopy}). As a consequence, the left derived functor $LF$ is faithful and the right derived functor $RU$ is essentially surjective (cf.\ Theorem \ref{LF_faithful}). We also give a detailed analysis of the cofibrations in $\Rep^\infty_k(G)$. Although this model category is not constructed from a cotorsion pair directly the cofibrant and the trivial objects satisfy a suitable orthogonality relation (cf.\ Lemma \ref{cofibrant_ext}). Finally, in Proposition \ref{characterization_equiv} we give several equivalent obstructions to the derived adjunction (\ref{LFRU_adjunction}) being an equivalence. The most interesting of these is related to the fact that the right transfer need no longer be {\it stable}. The necessary computation of the loop and suspension functors is described in Lemma \ref{cylinder_obj}.\\

However, the situation can be improved significantly. As a major result we show in \S\ref{section_5} that the functor $U:\Rep^\infty_k(G)\to\Mod(H)$ restricts to an equivalence between suitable full subcategories. On the one hand, we let $\mathrm{GProj}(H)$ denote the full subcategory of $\Mod(H)$ consisting of all Gorenstein projective $H$-modules. On the other hand, we define the relative notion of an {\it $I$-Gorenstein projective $G$-representation} (cf.\ Definition \ref{Cabanes_category}). This yields a full subcategory $\C(G)$ of $\Rep^\infty_k(G)$ such that the functor $U$ restricts to an equivalence $\C(G)\cong\mathrm{GProj}(H)$ (cf.\ Theorem \ref{Cabanes}). There is an explicit inverse functor $\FF$ whose relation to the functor $F$ is spelled out in Lemma \ref{relation_F_FF}. If the ring $H$ is selfinjective then $\mathrm{GProj}(H)=\Mod(H)$ and $\C(G)$ is the category of representations which are both a quotient and a subobject of an $I$-projective $G$-representation. In this case we recover a result of Cabanes concerning the case of finite groups with a split $BN$-pair of characteristic $p$ (cf.\ Remark \ref{Cabanes_finite_reductive}).\\

Endowed with the class of short $I$-exact sequences the category $\C(G)$ turns out to be a {\it weakly idempotent complete Frobenius category} (cf.\ Corollary \ref{C(G)_Frobenius}). Thus, $\C(G)$ has a canonical exact model structure and an associated homotopy category (cf.\ Proposition \ref{Frobenius_model_structure}). If $\Rep^\infty_k(G)$ is endowed with the right transfer studied in \S\ref{section_4} and if $i:\C(G)\to\Rep^\infty_k(G)$ denotes the inclusion functor then the composition
\[
\Ho(\C(G))\stackrel{\Ho(i)}{\longrightarrow}\Ho(\Rep^\infty_k(G))\stackrel{RU}{\longrightarrow}\Ho(\Mod(H))
\]
is an equivalence of categories (cf.\ Theorem \ref{derived_right_inverse}). In particular, $RU$ admits a right inverse and $\Ho(i)$ allows us to view $\Ho(\C(G))$ as a (not necessarily full) subcategory of $\Ho(\Rep^\infty_k(G))$ on which $RU$ becomes an equivalence. We hope this result amply demonstrates the strength of our approach.\\

At the end of our article we specialize to the case that $H$ is the pro-$p$ Iwahori-Hecke algebra of a split connected reductive group $\mathbb{G}$ over a non-archimedean local field of residue characteristic $p$. If $\mathbb{G}$ is semisimple we consider the canonical Gorenstein projective resolution $\GP_\bullet(M)\to M\to 0$ of an $H$-module $M$ constructed by Ollivier and Schneider (cf.\ \cite{OS14}, \S6.4). On the other hand, we consider the complex $\C_c^\mathrm{or}(\mathcal{X}_{(\bullet)},\F(M))$ of oriented chains of the $G$-equivariant coefficient system $\F(M)$ on the Bruhat-Tits building $\mathcal{X}$ of $G$ associated with $M$ by the second author (cf.\ \cite{Koh}, \S3.2). In Proposition \ref{oriented_chains} we show that there is a functorial isomorphism of complexes $\C^\mathrm{or}_c(\mathcal{X}_{(\bullet)},\F(M))\cong\FF\GP(M)_\bullet$. In particular, the chain complex of $\F(M)$ consists of objects of $\C(G)$. For the chain complex of a fixed point system as studied in \cite{OS14}, \S3.1, this is generally not true (cf.\ Example \ref{not_C(G)}).\\

{\bf Acknowledgments.} This work was partially funded by the DFG grant {\it Smooth modular representation theory of $p$-adic reductive groups}. Both authors are members of the DFG Research Training Group {\it Symmetries and Classifying Spaces -- Analytic, Arithmetic and Derived}. The financial support of the DFG is gratefully acknowledged. We would also like to thank the anonymous referee for a very careful reading and many valuable suggestions.\\

{\bf Notation and conventions.} A class of objects of a category $\C$ will usually be identified with the corresponding full subcategory. If we denote an adjunction by $F:\C\rightleftarrows\D:U$ then $F$ is always assumed to be left adjoint to $U$. The unit (resp.\ the counit) of an adjunction will always be denoted by $\eta$ (resp.\ $\varepsilon$). We say that a functor $F:\C\to\D$ {\it preserves} (resp.\ {\it reflects}) a property (P) if $F*$ (resp.\ $*$) has property (P) whenever $*$ (resp.\ $F*$) does. For any unital ring $S$ we denote by $\Mod(S)$ the category of left $S$-modules. Given $M,N\in\Mod(S)$ we write $\Hom_S(M,N)$ for the set of $S$-linear maps from $M$ to $N$.


\section{Preliminaries on model categories}\label{section_1}

%
%


We largely follow the conventions of Hovey's book \cite{HovBook}. A {\it model category} is a locally small and bicomplete category $\C$ endowed with a {\it model structure} consisting of three subcategories -- whose morphisms are called the {\it weak equivalences}, {\it cofibrations} and {\it fibrations}, respectively -- and two {\it functorial factorizations} as in \cite{HovBook}, Definition 1.1.3. As usual, a cofibration (resp.\ a fibration) is called {\it trivial} if it is also a weak equivalence.\\

An object $X\in\C$ is called {\it cofibrant} (resp.\ {\it fibrant}) if the map from the initial object to $X$ (resp.\ from $X$ to the terminal object) is a cofibration (resp.\ a fibration). We denote by $\C_c$ (resp.\ $\C_f$, resp.\ $\C_\cf$) the full subcategory of $\C$ consisting of all objects which are cofibrant (resp.\ fibrant, resp.\ cofibrant and fibrant). By
\[
Q_c:\C\to\C_c\qquad\mbox{and}\qquad Q_f:\C\to\C_f
\]
we denote the {\it cofibrant replacement functor} and the {\it fibrant replacement functor}, respectively (cf.\ \cite{HovBook}, page 5). The cofibrant replacement functor will often be abbreviated to $Q_c=Q$.\\

A \emph{path object} for an object $B\in\C$ is an object $P\in\C$ such that the diagonal morphism $\Delta:B\to B\times B$ admits a factorization $B\stackrel{i}{\longrightarrow}P\stackrel{p}{\longrightarrow}B\times B$ where $i$ is a weak equivalence and $p$ is a fibration. We note that while path objects can be constructed via the functorial factorizations of $\Delta$ the above definition allows any factorization with the required properties.\\

If $f:A\to B$ and $g:A\to B$ are morphisms in $\C$ then a \emph{right homotopy} from $f$ to $g$ is a morphism $H:A\to P$ from $A$ into a path object $B\stackrel{i}{\longrightarrow}P\stackrel{p}{\longrightarrow}B\times B$ for $B$ such that $p_1 p H=f$ and $p_2 p H=g$. Here $p_1,p_2:B\times B\to B$ denote the two projections. The morphisms $f$ and $g$ are called {\it right homotopic} if there is a right homotopy from $f$ to $g$. \\

There are also dual notions of {\it cylinder objects} and {\it left homotopies} as in \cite{HovBook}, Definition 1.2.4. Morphisms which are both right and left homotopic are called {\it homotopic}. If $X\in\C_c$ and $Y\in\C_f$ then two morphisms $f,g:X\to Y$ are right homotopic if and only if they are left homotopic (cf.\ \cite{HovBook}, Proposition 1.2.5 (v)). Moreover, on the morphisms of $\C_\cf$ the homotopy relation $\sim$ is an equivalence relation compatible with composition (cf.\ \cite{HovBook}, Corollary 1.2.7). Dividing out the homotopy relation on the level of morphisms leads to a category denoted $\C_\cf/\!\sim$.\\

The {\it homotopy category} $\Ho(\C)$ of a model category $\C$ is the localization of $\C$ with respect to the class of weak equivalences (cf.\ \cite{HovBook}, Definition 1.2.1). We will usually just write $\C\to\Ho(\C)$ for the canonical functor of $\C$ into its homotopy category. Any subcategory of $\C$ has an induced class of weak equivalences and the corresponding homotopy category is defined analogously. It is a fundamental result that the inclusion $\C_\cf\to\C$ induces equivalences of categories
\begin{equation}\label{main_thm_model_cat}
\C_\cf/\!\sim\;\stackrel{\cong}{\longrightarrow}\;\Ho(\C_\cf)\;\stackrel{\cong}{\longrightarrow}\;\Ho(\C)
\end{equation}
(cf.\ \cite{HovBook}, Theorem 1.2.10 (i)). Given two objects $X,Y\in\C$ the set of morphisms from $X$ to $Y$ in $\Ho(\C)$ is usually denoted by $[X,Y]$.\\

Recall that a functor $F:\C\to\D$ (resp.\ $U:\D\to\C$) between model categories is called \emph{left Quillen} (resp.\ \emph{right Quillen}) if it admits a right (resp.\ left) adjoint and preserves cofibrations and trivial cofibrations (resp.\ fibrations and trivial fibrations). If $F:\C\rightleftarrows\D:U$ is an adjunction between model categories then $F$ is left Quillen if and only if $U$ is right Quillen (cf.\ \cite{HovBook}, Lemma 1.3.4). In this case we speak of a \emph{Quillen adjunction}. A Quillen adjunction is called a \emph{Quillen equivalence} if for all cofibrant $X\in\C$ and all fibrant $Y\in\D$ a morphism $FX\to Y$ is a weak equivalence in $\D$ if and only if the adjoint morphism $X\to UY$ is a weak equivalence in $\C$ (cf.\ \cite{HovBook}, Definition 1.3.12).\\

We refer to \cite{HovBook}, Proposition 1.3.13 and Corollary 1.3.16, for various characterizations of when a Quillen adjunction is a Quillen equivalence. At this point we just recall how this can be seen on the level of homotopy categories. Note that the functors
\[
\C\stackrel{Q_c}{\longrightarrow}\C_c\longrightarrow\Ho(\C_c)\qquad\mbox{and}\qquad\C_c\stackrel{F}{\longrightarrow}\D\longrightarrow\Ho(\D)
\]
send weak equivalences to isomorphisms. Indeed, for the left functor this follows from the 2-out-of-3 property for weak equivalences. For the right functor this follows from \cite{HovBook}, Lemma 1.1.12. By the universal property of localizations there are induced functors $\Ho(Q_c):\Ho(\C)\to\Ho(\C_c)$ and $\Ho(F):\Ho(\C_c)\to\Ho(\D)$. The {\it left derived functor} $LF:\Ho(\C)\to\Ho(\D)$ of $F$ is then defined to be the composite
\begin{equation}\label{left_derived}
\xymatrix{
\Ho(\C)\ar[r]^{\Ho(Q_c)}&\Ho(\C_c)\ar[r]^{\Ho(F)}& \Ho(\D)
}
\end{equation}
(cf.\ \cite{HovBook}, Definition 1.3.6). Similarly, one defines the {\it right derived functor} $RU:\Ho(\D)\to\Ho(\C)$ of $U$ as the composite
\begin{equation}\label{right_derived}
\xymatrix{
\Ho(\D)\ar[r]^{\Ho(Q_f)}&\Ho(\D_f)\ar[r]^{\Ho(U)}&\Ho(\C).
}
\end{equation}
In this situation $LF$ is left adjoint to $RU$ (cf.\ \cite{HovBook}, Lemma 1.3.10). Moreover, $F$ and $U$ form a Quillen equivalence if and only if $LF$ and $RU$ are inverse equivalences of categories (cf.\ \cite{HovBook}, Proposition 1.3.13).\\

A powerful tool to construct model structures on an abelian category $\A$ was developed by Hovey in \cite{Hov2}. It is based on a fixed {\it proper class} $\mathcal{P}$ of short exact sequences in $\A$ leading to relative extension groups $\Ext_\mathcal{P}^\bullet(X,Y)$ as in \cite{Mac}, Chapter XII.4. An epimorphism $f:X\to Y$ in $\A$ is called a {\it $\mathcal{P}$-epimorphism} if the short exact sequence
\[
0\longrightarrow\ker(f)\longrightarrow X\stackrel{f}{\longrightarrow}Y\longrightarrow 0
\]
belongs to $\mathcal{P}$. The notion of a {\it $\mathcal{P}$-monomorphism} is defined dually. Given classes $\C$ and $\F$ of objects of $\A$ consider the classes of objects
\[
^\perp\F=\{X\in\A\;|\;\Ext^1_\mathcal{P}(X,Y)=0\mbox{ for all }Y\in\F\}
\]
and
\[
\C^\perp=\{Y\in\A\;|\;\Ext^1_\mathcal{P}(X,Y)=0\mbox{ for all }X\in\C\}.
\]
A {\it cotorsion pair} with respect to $\mathcal{P}$ is a pair $(\C,\F)$ of classes of objects of $\A$ satisfying $\C={^\perp\F}$ and $\C^\perp=\F$ (cf.\ \cite{Hov2}, Definition 2.3). It is said to have {\it enough functorial projectives} if for any $X\in\A$ there is a $\mathcal{P}$-epimorphism $f:Y\to X$ which is functorial in $X$ and satisfies $Y\in\C$ and $\ker(f)\in\F$. The existence of {\it enough functorial injectives} is defined dually. A cotorsion pair which has both enough functorial projectives and enough functorial injectives is called {\it functorially complete}.\\

For the definition of a model structure on $\A$ which is compatible with $\mathcal{P}$ we refer to \cite{Hov2}, Definition 2.1. A subcategory of $\A$ is called {\it $\mathcal{P}$-thick} if it is closed under retracts and satisfies the 2-out-of-3 property with respect to short exact sequences in $\mathcal{P}$. If $\A$ is an abelian category endowed with the structure of a model category then an object $X\in\A$ is called {\it trivial} if the morphism $0\to X$ (or equivalently the morphism $X\to 0$) is a weak equivalence. We denote by $\A_t$ the full subcategory of $\A$ consisting of all trivial objects. The following fundamental theorem is due to Hovey (cf.\ \cite{Hov2}, Theorem 2.2).

\begin{thm}[Hovey]\label{cotorsion_pairs}
Let $\A$ be a bicomplete abelian category with a fixed proper class $\mathcal{P}$ of short exact sequences.
\begin{itemize}
\item[(i)]If $\A$ carries a model structure which is compatible with $\mathcal{P}$ then $\A_t$ is $\mathcal{P}$-thick and $(\A_c,\A_t\cap\A_f)$ and $(\A_c\cap\A_t,\A_f)$ are functorially complete cotorsion pairs.
\item[(ii)]If $\C$, $\F$ and $\T$ are full subcategories of $\A$ such that $\T$ is $\mathcal{P}$-thick and such that $(\C\cap\T,\F)$ and $(\C,\T\cap\F)$ are functorially complete cotorsion pairs then there is a unique model structure on $\A$ which is compatible with $\mathcal{P}$ and which satisfies $\C=\A_c$, $\F=\A_f$ and $\T=\A_t$. \hfill$\Box$
\end{itemize}
\end{thm}

In the situation of Theorem \ref{cotorsion_pairs} (ii) the definition of the corresponding cofibrations, fibrations and weak equivalences is given explicitly in \cite{Hov2}, Definition 5.1. If $\mathcal{P}$ is the class of all short exact sequences then a model structure (resp.\ a model category) as in Theorem \ref{cotorsion_pairs} is called {\it abelian}. An abelian model structure is called {\it hereditary} if the class of cofibrant objects is closed under taking kernels of epimorphisms and if the class of fibrant objects is closed under taking cokernels of monomorphisms.\\

For simplicity we will often assume that all objects of our categories are {\it small} (cf.\ \cite{Hov1}, Definition A.1). For example, this holds in any Grothendieck category (cf.\ \cite{Hov1}, Proposition A.2) which is the main case of interest for us. We refer to \cite{Hov2}, Corollary 6.8, for a manageable criterion to check whether a given cotorsion pair on a Grothendieck category is functorially complete.\\

Let $\mathcal{I}$ be a class of morphisms in a category $\C$. Recall that a morphism in $\C$ is called {\it $\mathcal{I}$-injective} if it satisfies the {\it right lifting property} with respect to all morphisms in $\mathcal{I}$ in the sense of \cite{HovBook}, Definition 1.1.2. The class of these morphisms is denoted by $\mathcal{I}$-inj. A morphism in $\C$ is called an {\it $\mathcal{I}$-cofibration} if it satisfies the {\it left lifting property} with respect to all morphisms in $\mathcal{I}$-inj. The class of these morphisms is denoted by $\mathcal{I}$-cof.\\

For the definition of a model category $\C$ which is {\it cofibrantly generated} we refer to \cite{HovBook}, Definition 2.1.17. If every object of $\C$ is small then this means that there are sets $\mathcal{I}$ and $\mathcal{J}$ of morphisms of $\C$ such that the class of fibrations is equal to $\mathcal{J}$-inj and the class of trivial fibrations is equal to $\mathcal{I}$-inj. The elements of the sets $\mathcal{I}$ and $\mathcal{J}$ are called the {\it generating cofibrations} and the {\it generating trivial cofibrations} of $\C$, respectively.\\

Assume that we have an adjunction $F:\C\rightleftarrows\D:U$. If $\C$ is a model category and if $\D$ is bicomplete then there is a natural candidate for an associated model structure on $\D$. Namely, define a fibration (resp.\ a weak equivalence, resp.\ a trivial fibration) in $\D$ to be a morphism $f$ in $\D$ such that $Uf$ is a fibration (resp.\ a weak equivalence, resp.\ a trivial fibration) in $\C$. Moreover, define a cofibration in $\D$ to be a morphism which satisfies the left lifting property with respect to all trivial fibrations. If these three classes of morphisms form part of a model structure on $\D$ then this is called the \emph{right transfer} and we say that the right transfer exists. We note that the term {\it right induction} is common, as well.\\

In order to discuss the existence of the right transfer we call a morphism in $\D$ \emph{anodyne} if it satisfies the left lifting property with respect to all fibrations. If the right transfer exists then an anodyne morphism is automatically a weak equivalence (cf.\ \cite{HovBook}, Lemma 1.1.10). In certain situations this condition is also sufficient. Indeed, we have the following result which is a variant of \cite{Hirsch}, Theorem 11.3.2.

\begin{thm}\label{right_transfer} Let $F:\C\rightleftarrows\D:U$ be an adjunction where $\C$ is a cofibrantly generated model category with generating cofibrations $\mathcal{I}$ and generating trivial cofibrations $\mathcal{J}$. Moreover, assume that $\D$ is bicomplete and that all objects of $\D$ are small. If every anodyne morphism in $\D$ is a weak equivalence then the right transfer exists. It is cofibrantly generated with generating cofibrations $F\mathcal{I}$ and generating trivial cofibrations $F\mathcal{J}$.
\end{thm}

\begin{proof}
If $f$ is a morphism in $\D$ then $Uf\in\mathcal{I}$-inj (resp.\ $Uf\in\mathcal{J}$-inj) if and only if $f\in F\mathcal{I}$-inj (resp.\ $f\in F\mathcal{J}$-inj) by the adjunction. Therefore, the class of fibrations in $\D$ is $F\mathcal{J}$-inj, the class of trivial fibrations is $F\mathcal{I}$-inj, the class of cofibrations is $F\mathcal{I}$-cof and the class of anodyne morphisms is $F\mathcal{J}$-cof.\\

It now suffices to check that the conditions in \cite{Hirsch}, Theorem 11.3.1, are satisfied for $F\mathcal{I}$, $F\mathcal{J}$ and the class of weak equivalences. Note first that the latter has the 2-out-of-3 property and is closed under retracts because $U$ is a functor. Conditions (1)--(3) and (4)(b) in \cite{Hirsch}, Theorem 11.3.1, follow from the fact that every object in $\D$ is small, that every anodyne morphism is a weak equivalence and that the morphisms in $F\mathcal{I}$-inj are exactly the trivial fibrations.
\end{proof}

\begin{rem} The idea of the right transfer goes back to the original work of Quillen (cf.\ \cite{Quillen}, section II.4). There are more general versions of the above theorem replacing the smallness assumptions on $\D$ by suitable properties of the functor $F$. See \cite{Hirsch}, Theorem 11.3.2, for instance.
\end{rem}

It might still be non-trivial to check that anodyne morphisms are weak equivalences. However, there is a simple condition under which this is true. Indeed, assume that  $\C$ is a model category and that we have an adjunction $F:\C\rightleftarrows\D:U$ where $\D$ is bicomplete. Since we have notions of fibrations and weak equivalences in $\D$ we can make sense of path objects and right homotopies in $\D$, as well. Moreover, since $U$ preserves products, fibrations and weak equivalences it also preserves path objects and right homotopies. Due to a lack of reference we include a proof of the following variant of \cite{GS07}, Theorem 3.8.
\begin{prop}\label{path_objects}
Assume that  $\C$ is a model category and that we have an adjunction $F:\C\rightleftarrows\D:U$ where $\D$ is bicomplete. If every object in $\C$ is fibrant and if every object of $\D$ admits a path object then every anodyne morphism in $\D$ is a weak equivalence.
\end{prop}

\begin{proof}
First note that every object in $\D$ is fibrant because $U$ preserves the terminal object $*$. Let $j: A\to B$ be an anodyne morphism in $\D$. By applying the left lifting property to the square
\[
\begin{tikzcd}
A\arrow{r}{\id_A}\arrow{d}{j} & A\arrow{d}\\
B\arrow[swap]{r} & *
\end{tikzcd}
\]
we get a morphism $u:B\to A$ such that $u j=\id_A$. Consider $f=j u:B\to B$ and let $P$ be a path object for $B$ with factorization $\Delta=pi$ into a fibration $p$ and a weak equivalence $i$. Since $u j=\id_A$ we have $f j=j$. Therefore, the outer square in the diagram
\[
\begin{tikzcd}
A\arrow{r}{ij}\arrow{d}{j} & P\arrow{d}{p}\\
B\arrow[swap]{r}{\id_B\times f}\arrow[dashrightarrow]{ur}{H} & B\times B
\end{tikzcd}
\]
commutes. The left lifting property implies the existence of $H:B\to P$ making the two triangles commute. By construction, $H$ is a right homotopy from $\id_B$ to $f$. Since $U$ preserves right homotopies $\id_{UB}$ is right homotopic to $Uf=Uj Uu$ in $\C$. Since right homotopic morphisms are identified in Ho$(\C)$ (cf.\ \cite{HovBook}, Theorem 1.2.10 (iii)) $Uu$ is right inverse to $Uj$ in Ho$(\C)$. Of course, it is also left inverse because this already holds in $\C$. Consequently, $Uj$ becomes an isomorphism in $\Ho(\C)$, hence is a weak equivalence in $\C$ by \cite{HovBook}, Theorem 1.2.10 (iv). Thus, $j$ is a weak equivalence in $\D$.
\end{proof}

\begin{rem}The result in Proposition \ref{path_objects} is again implicit in the work of Quillen \cite{Quillen} who works in the simplicial setting. The condition that every object is fibrant can be replaced by the more general condition that $\D$ admits a functorial fibrant replacement and functorial path objects (cf.\ \cite{GS07}, Theorem 3.8). The above proof is a rewording of its simplicial version in \cite{GJ99}, Lemma II.6.1.
\end{rem}

\begin{lem}\label{unit_equiv}
Suppose that we have an adjunction $F:\C\rightleftarrows\D:U$ where $\D$ is bicomplete and $\C$ is a model category whose right transfer exists. Then $F$ and $U$ form a Quillen adjunction. This is a Quillen equivalence if and only if the unit $\eta_X: X\to UFX$ is a weak equivalence for all cofibrant $X$ in $\C$.
\end{lem}

\begin{proof}
It follows immediately from the definition of the right transfer that the functor $U$ is right Quillen. Now assume that $\eta_X$ is a weak equivalence for all cofibrant $X$ in $\C$. Consider a morphism $f:FX\to Y$ where $X\in \C$ is cofibrant and $Y\in \D$ is fibrant. By definition of the right transfer $f:FX\to Y$ is a weak equivalence if and only if so is $Uf$. By the 2-out-of-3 property this is true if and only if $Uf \eta_X$ is a weak equivalence. However, the latter is the adjoint of $f$. Thus, we have a Quillen equivalence.\\

Conversely, if we have a Quillen equivalence and if $X$ is cofibrant in $\C$ consider the fibrant replacement $f:FX\to Q_fFX$ of $FX$ in $\D$. Since this is a weak equivalence so are $Uf$ and the adjoint morphism $Uf\eta_X$. The 2-out-of-3 property implies that also $\eta_X$ is a weak equivalence.
\end{proof}

A bicomplete category is called {\it pointed} if the morphism from the initial to the terminal object is an isomorphism. The homotopy category of any pointed model category $\C$ admits two endofunctors $\Sigma$ and $\Omega$ called the {\it suspension functor} and the {\it loop functor}, respectively. Given $X\in \C$ fibrant and $P$ a path object for $X$, the loop $\Omega X$ of $X$ is defined to be the pullback of the diagram
$$
P\stackrel{p}{\longrightarrow}X\times X\longleftarrow *
$$
where $*$ denotes the initial/terminal object of $\C$. Given $X\in \C$ cofibrant, its suspension $\Sigma X$ is defined dually using a cylinder object (cf.\ \cite{BR}, Definition 3.1.2). The suspension functor $\Sigma$ is left adjoint to the loop functor $\Omega$ (cf.\ \cite{BR}, Proposition 3.1.7). If they are equivalences of categories then the model category $\C$ is called {\it stable}. If $F:\C\rightleftarrows\mathcal{D}:U$ is a Quillen adjunction then $LF$ commutes with $\Sigma$ and $RU$ commutes with $\Omega$ (cf.\ \cite{BR}, Corollary 3.1.4). If both model categories are stable then also $RU$ commutes with $\Sigma$ and $LF$ commutes with $\Omega$ (cf.\ \cite{BR}, Corollary 3.2.10).\\

To compute the suspension of an object $X$ in an abelian model category $\A$ choose a trivial fibration $Y\to X$ with $Y$ cofibrant. If $0\to Y\to Y'\to Y''\to 0$ is an exact sequence with $Y'$ trivially fibrant and $Y''$ cofibrant then $\Sigma X\cong Y''$ in $\Ho(\A)$. For the loop functor one chooses a trivial cofibration $X\to Y$ with $Y$ fibrant. If $0\to Y''\to Y'\to Y\to 0$ is an exact sequence in $\C$ with $Y'$ trivially cofibrant and $Y''$ fibrant then $\Omega X\cong Y''$ in $\Ho(\A)$. If an abelian model structure is hereditary then it is also stable (cf.\ \cite{Beck14}, Corollary 1.1.15). \\

Finally, we give two standard examples of a model category that will play a role in the following. Given an abelian category $\A$ we denote by $\Ch(\A)$ the category of unbounded chain complexes 
\[
X=X_\bullet=(X_\bullet,d_\bullet)=(\ldots\stackrel{d_2}{\longrightarrow}A_1
\stackrel{d_1}{\longrightarrow}A_0\stackrel{d_0}{\longrightarrow}A_{-1}
\stackrel{d_{-1}}{\longrightarrow}\ldots)
\]
over $\A$. Let $\iota_0:\A\to\Ch(\A)$ denote the functor obtained by viewing an object of $\A$ as a complex concentrated in degree zero. We will often use it to view $\A$ as a full subcategory of $\Ch(\A)$ and suppress it from the notation. Recall that $\iota_0$ admits both a left adjoint $Q_0$ and a right adjoint $Z_0$ given by
\begin{equation}\label{left_right_adjoints}
Q_0X=\coker(d_1)\quad\mbox{and}\quad Z_0X=\ker(d_0).
\end{equation}
For any integer $n$ we denote by $X\mapsto B_nX=\im(d_{n+1})$ the $n$-th boundary functor, by $X\mapsto Z_nX=\ker(d_n)$ the $n$-th cycle functor and by $X\mapsto H_nX=Z_nX/B_nX$ the $n$-th homology functor $\Ch(\A)\to\A$. Let $[1]$ denote the shift functor on $\Ch(\A)$ defined by $(X_\bullet,d_\bullet)[1]=(Y_\bullet,e_\bullet)$ with $Y_n=X_{n+1}$ and $e_n=-d_{n+1}$ for all $n\in\Z$. We denote by $[-1]$ the two-sided inverse of $[1]$.\\

Note that together with $\A$ also the category $\Ch(\A)$ is abelian. Moreover, we have the following standard facts.

\begin{lem}\label{complexes_Grothendieck}
\begin{enumerate}[wide]Let $\A$ be an abelian category.
\item[(i)]If $\A$ is a Grothendieck category then so is $\Ch(\A)$.
\item[(ii)]If $\A$ has enough projectives then so does $\Ch(\A)$.
\end{enumerate}
\end{lem}

\begin{proof}
The abelian category $\Ch(\A)$ is cocomplete by defining colimits term\-wise. Filtered colimits are exact because kernels and cokernels are defined termwise, as well. If $U$ is a generator of $\A$ consider the complex $D^nU=[0\to U\stackrel{\id}{\to}U\to 0]$ in $\Ch(\A)$ with $U$ in degrees $n$ and $n-1$. Then $\bigoplus_{n\in\Z}D^nU$ is a generator of $\Ch(\A)$. This shows (i).\\

As for (ii) let $\A^\Z$ denote the category of $\Z$-graded objects of $\A$ and consider the functor $T:\Ch(\A)\to\A^\Z$ given by $X_\bullet\mapsto \bigoplus_{n\in\Z}X_n$. It has the left adjoint $D$ sending $\bigoplus_{n\in\Z}X_n$ to $\bigoplus_{n\in\Z}D^nX_n$. In degree $n$ the map of complexes
\[
D^nX_n\hookrightarrow DTX_\bullet\stackrel{\varepsilon_{X_\bullet}}{\longrightarrow}X_\bullet
\]
is the identity on $X_n$. This implies that the counit $\varepsilon_{X_\bullet}:DTX_\bullet\to X_\bullet$ is an epimorphism. Therefore, given $X_\bullet\in\Ch(\A)$ it suffices to construct an epimorphism $f:Q_\bullet\to DTX_\bullet$ where $Q_\bullet\in\Ch(\A)$ is projective. For any $n$ we choose an epimorphism $Y_n\to X_n$ in $\A$ where $Y_n$ is projective. Consider the epimorphism $g:\bigoplus_{n\in\Z}Y_n\to\bigoplus_{n\in\Z}X_n$ in $\A^\Z$ and note that $\bigoplus_{n\in\Z}Y_n$ is a projective object in $\A^\Z$. We claim that the induced map
\[
f=Dg:Q_\bullet=D(\bigoplus_{n\in\Z}Y_n)\longrightarrow D(\bigoplus_{n\in\Z}X_n)=DTX_\bullet
\]
in $\Ch(\A)$ is as required. Indeed, $f$ is an epimorphism because $D$ is a left adjoint and so preserves epimorphisms. Moreover, $Q_\bullet$ is projective because the right adjoint $T$ is exact and hence $D$ preserves projectives.
\end{proof}

Recall that any Grothendieck category is automatically complete and therefore bicomplete. If $\A$ is a Grothendieck category there is a cofibrantly generated model structure on $\Ch(\A)$ for which the cofibrations are the monomorphisms and the weak equivalences are the quasi-isomorphisms (cf.\ \cite{Beke}, Proposition 3.13). This is called the {\it injective model structure} on $\Ch(\A)$. We write $\Ch(\A)^\Inj$ for the corresponding model category. Its homotopy category is equivalent to the unbounded derived category $D(\A)\cong\Ho(\Ch(\A)^\Inj)$ of $\A$.\\

If $\A$ is the category of modules over a ring then the injective model structure on $\Ch(\A)$ is also constructed in \cite{HovBook}, \S2.3. It follows as in \cite{HovBook}, Proposition 2.3.20, that the fibrations are precisely the termwise split surjections with fibrant kernels. Looking at the long exact homology sequence we see that the trivial fibrations are precisely the termwise split surjections with trivially fibrant kernels. Consequently, the injective model structure on $\Ch(\A)$ is compatible with the class $\mathcal{P}$ of all short exact sequences.\\

By Theorem \ref{cotorsion_pairs} the injective model structure is related to two functorially complete cotorsion pairs on $\Ch(\A)$ given as follows. Note that every object of $\Ch(\A)^\Inj$ is cofibrant, i.e.\ we have $\C=\Ch(\A)^\Inj_c=\Ch(\A)$. Moreover, the class $\T=\Ch(\A)^\Inj_t$ of trivial objects is precisely that of the acyclic complexes. If we denote by $\F=\Ch(\A)^\Inj_f$ the class of fibrant objects then we know from Theorem \ref{cotorsion_pairs} that $\F=\T^\perp$. In other words, a complex $Y$ is fibrant if and only if $\Ext_{\Ch(\A)}^1(X,Y)=0$ for any acyclic complex $X$. As in \cite{EJX}, Proposition 3.4, one shows that this condition is satisfied if and only if $Y$ is a {\it dg-injective} complex, i.e.\ $Y$ is termwise injective and the internal hom complex $\underline{\Hom}(X,Y)$ is acyclic for any acyclic complex $X$.\\

Apart from the cotorsion pair $(\C\cap\T,\F)=(\T,\T^\perp)$ there is also the cotorsion pair $(\C,\T\cap\F)=(\Ch(\A),\T\cap\T^\perp)$. As in \cite{EJX}, Proposition 3.7, the class $\T\cap\T^\perp$ is that of the {\it injective} complexes, i.e.\ of the injective objects of $\Ch(\A)$. Recall that these are precisely the split acyclic complexes which are termwise injective.

\begin{rem}\label{Spaltenstein}
There is also the notion of a {\it $K$-injective} complex in the sense of Spaltenstein (cf.\ \cite{Spa}, page 124). Recall that a complex $Y\in\Ch(\A)$ is $K$-injective if and only if the complex $\underline{\Hom}(X,Y)$ is acyclic for any acyclic complex $X$. As mentioned above, a complex is dg-injective if and only if it is $K$-injective and termwise injective. Consequently, many of Spaltenstein's results in \cite{Spa} can be reproved using the injective model structure on $\Ch(\A)$. For example, the fibrant replacement $Y\to Q_fY$ of $Y\in\Ch(\A)$ is an injective quasi-isomorphism such that $Q_fY$ is dg-injective.
\end{rem}

Under suitable assumptions on $\A$ there is a projective version of the above results. For our purposes it will be sufficient to stick to the classical case that $\A=\Mod(S)$ is the category of left modules over a ring $S$. In this case we write $\Ch(S)=\Ch(\Mod(S))$. According to \cite{HovBook}, Theorem 2.3.11, there is a cofibrantly generated model structure on $\Ch(S)$ for which the fibrations are the surjections and the weak equivalences are the quasi-isomorphisms. This is called the {\it projective model structure} on $\Ch(S)$. We write $\Ch(S)^\Proj$ for the corresponding model category. Its homotopy category is again equivalent to the unbounded derived category $D(S)\cong\Ho(\Ch(S)^\Proj)$ of $\Mod(S)$.\\

By \cite{HovBook}, Proposition 2.3.9, the cofibrations are precisely the termwise split monomorphisms with cofibrant cokernel. It follows that the projective model structure on $\Ch(S)$ is compatible with the class $\mathcal{P}$ of all short exact sequences. By Theorem \ref{cotorsion_pairs} it is related to two functorially complete cotorsion pairs on $\Ch(S)$. Note that every object of $\Ch(S)^\Proj$ is fibrant, i.e.\ we have $\F=\Ch(S)^\Proj_f=\Ch(S)$. Moreover, the class $\T=\Ch(S)^\Proj_t$ of trivial objects is precisely that of the acyclic complexes. If we denote by $\C=\Ch(S)^\Proj_c$ the class of cofibrant objects then we know from Theorem \ref{cotorsion_pairs} that $\C={^\perp\T}$. In other words, a complex $X$ is cofibrant if and only if $\Ext_{\Ch(S)}^1(X,Y)=0$ for any acyclic complex $Y$. As in \cite{EJX}, Proposition 3.5, one shows that this condition is satisfied if and only if $X$ is a {\it dg-projective} complex, i.e.\ $X$ is termwise projective and the complex $\underline{\Hom}(X,Y)$ is acyclic for any acyclic complex $Y$.\\

Apart from the cotorsion pair $(\C,\T\cap\F)=({^\perp\T},\T)$ there is also the cotorsion pair $(\C\cap\T,\F)=({^\perp\T}\cap\T,\Ch(S))$. As in \cite{EJX}, Proposition 3.7, the class ${^\perp\T}\cap\T$ is that of the {\it projective} complexes, i.e.\ of the projective objects of $\Ch(S)$. Recall that these are precisely the split acyclic complexes which are termwise projective. In Spaltenstein's terminology a complex is dg-projective if and only if it is $K$-projective and termwise projective.\\

Finally, it is clear from the above that $\id:\Ch(S)^\Proj\rightleftarrows\Ch(S)^\Inj:\id$ is a Quillen adjunction. Of course, it is even a Quillen equivalence because the classes of weak equivalences coincide. For both the projective and the injective model structure on $\Ch(S)$ the suspension $\Sigma=L[-1]$ is the left derived functor of the shift functor $[-1]$. Likewise, the loop functor $\Omega=R[1]$ is the right derived functor of the shift $[1]$. In a different setting an anologous result will be proven in Lemma \ref{explicit_loop_susp} below.


\section{Derived categories and relative homological algebra}\label{section_2}

%
%


Recall the following notions from relative homological algebra (cf.\ \cite{ChHov}, \S1). We continue to denote by $\A$ an abelian category and fix a class $\mathcal{P}$ of objects of $\A$. Given any object $P$ of $\A$ we say that a morphism $f:A\to B$ in $\A$ is \emph{$P$-epic} or a \emph{$P$-epimorphism} if the induced map
$$
\Hom_{\A}(P,A)\to \Hom_{\A}(P,B)
$$
is surjective. Furthermore, we say that $f$ is \emph{$\mathcal{P}$-epic} if it is $P$-epic for all $P\in\mathcal{P}$. A sequence $A\overset{f}{\longrightarrow}B\overset{g}{\longrightarrow}C$ is called \emph{$P$-exact} if $gf=0$ and if the induced sequence $\Hom_{\A}(P,A)\to\Hom_{\A}(P,B)\to \Hom_{\A}(P,C)$ of abelian groups is exact. The sequence $A\to B\to C$ is called $\mathcal{P}$-exact if it is $P$-exact for all $P\in\mathcal{P}$.

\begin{defin}\label{proj_class} A \emph{projective class} on $\A$ is a pair $(\mathcal{P},\mathcal{E})$ where $\mathcal{P}$ is a class of objects of $\A$ and $\mathcal{E}$ is a class of morphisms in $\A$ such that the following three conditions are satisfied.
\begin{enumerate}
\item[(i)] $\mathcal{E}$ is the class of all $\mathcal{P}$-epimorphisms.
\item[(ii)]  $\mathcal{P}$ is the class of all objects $P$ of $\A$ such that every morphism in $\mathcal{E}$ is $P$-epic.
\item[(iii)] For each object $B$ in $\A$ there is a morphism $f:P\to B$ in $\A$ with $P\in\mathcal{P}$ and $f\in\mathcal{E}$.
\end{enumerate}
\end{defin}

\begin{ex}\label{standard_proj_class} For any ring $S$ let $\mathcal{P}$ be the class of all projective $S$-modules and let $\mathcal{E}$ be the class of all surjective $S$-linear maps. Then the pair $(\mathcal{P},\mathcal{E})$ is a projective class on $\A=\Mod(S)$ called the {\it standard projective class}.
\end{ex}

A {\it $\mathcal{P}$-resolution} of an object $B\in\A$ is a $\mathcal{P}$-exact sequence $P_\bullet\to B\to 0$ in $\A$ such that $P_n\in\mathcal{P}$ for all $n\geq 0$. Given another object $C\in\A$ the group of relative extensions of $C$ and $B$ are defined by
\begin{equation}\label{relative_Ext}
\Ext_\mathcal{P}^n(B,C)=H^n\Hom_\A(P_\bullet,C)\quad\mbox{for all }n\geq 0.
\end{equation}
Up to isomorphism they are independent of the chosen $\mathcal{P}$-resolution of $B$ and any short $\mathcal{P}$-exact sequence $0\to B'\to B\to B''\to 0$ gives rise to a long exact sequence between relative extension groups, as usual.

\begin{rem}\label{relative_ext_ambiguous}
In spite of the notational similarities the relative extension groups in (\ref{relative_Ext}) are not to be confused with those defining cotorsion pairs. In fact, we shall see in examples below that short $\mathcal{P}$-exact sequences are generally not exact in the usual sense and hence do not belong to any proper class of short exact sequences in $\A$. 
\end{rem}

Suppose now that we have an adjunction $F:\mathcal{B}\rightleftarrows\A:U$ between abelian categories. Moreover, assume that $(\mathcal{P}',\mathcal{E}')$ is a projective class on $\mathcal{B}$. Recall that a \emph{retract} of an object $B\in\A$ is an object $A\in\A$ admitting a right invertible morphism $B\to A$. Since $\A$ is abelian this is equivalent to $A$ being isomorphic to a direct summand of $B$. We define the \emph{pullback} of the pair $(\mathcal{P}',\mathcal{E}')$ along $U$ to be the pair $(\mathcal{P},\mathcal{E})$ where
\begin{itemize}
\item $\mathcal{P}=\{\text{retracts of }FP\;|\;P\in \mathcal{P}'\}$ and
\item $\mathcal{E}=\{f:B\to C\;|\;Uf\in\mathcal{E}'\}$.
\end{itemize}

The proof of the following result is straightforward and left to the reader.

\begin{lem}\label{pullback} With the above notation, the pullback $(\mathcal{P},\mathcal{E})$ is a projective class on $\A$. Moreover, a sequence $A\to B\to C$ in $\A$ is $\mathcal{P}$-exact if and only if it is a complex and the sequence $UA\to UB\to UC$ in $\mathcal{B}$ is $\mathcal{P}'$-exact.\hfill$\Box$
\end{lem}

For the rest of this article we fix a field $k$ of positive characteristic $p$ and denote by $G$ a topological group admitting an open subgroup $I$ which is pro-$p$. Let $\Rep^\infty_k(G)$ denote the category of {\it smooth} $k$-linear representations of $G$, i.e.\ of $k$-vector spaces endowed with a $k$-linear action of $G$ such that the stabilizer of any vector is open in $G$. Its morphisms are the $k$-linear and $G$-equivariant maps. Given $V,W\in\Rep^\infty_k(G)$ we will write $\Hom_G(V,W)$ for the set of morphisms from $V$ to $W$ in $\Rep^\infty_k(G)$. We explicitly allow $G$ to be discrete. In this case the category $\Rep^\infty_k(G)=\Mod(k[G])$ is that of all $k$-linear representations of $G$.\\

Given $V\in\Rep^\infty_k(I)$ we denote by $\ind_I^G(V)\in\Rep^\infty_k(G)$ the {\it compact induction of $V$} from $I$ to $G$, i.e.\ the $k$-vector space of all compactly supported maps $f:G\to V$ satisfying $f(gi)=i^{-1}f(g)$ for all $i\in I$ and $g\in G$. The $G$-action on $\ind_I^G(V)$ is given by $(gf)(g')=f(g^{-1}g')$ for all $g,g'\in G$. We endow $k$ with the trivial action of $I$ and set $\XX =\ind_I^G(k)\in\Rep^\infty_k(G)$.\\

Setting $H=\End_G(\XX )^\op$ Frobenius reciprocity shows that we have adjoint functors
\begin{equation}\label{Frobenius_reciprocity}
F:\Mod(H)\rightleftarrows \Rep_k^\infty(G):U
\end{equation}
given by $FM=\XX \otimes_HM$ and $UV=V^I\cong\Hom_G(\XX ,V)$. Since $k$ has characteristic $p$ the functor $U$ has the property that $UV\neq 0$ for all $0\neq V\in\Rep_k^\infty(G)$  (cf.\ \cite{Sch}, Lemma 1 (v)). Since $U$ is also left exact it reflects monomorphisms.\\

For $V\in \Rep_k^\infty(G)$ the counit of the adjunction (\ref{Frobenius_reciprocity}) is the $G$-equivariant map
\begin{equation}\label{counit}
\varepsilon_V: FUV=\XX \otimes_H V^I \to V,\quad
1_{gI}\otimes v \mapsto g v,
\end{equation}
where $1_{gI}$ denotes the characteristic function of the coset $gI$. Moreover, given $M\in \Mod(H)$ the unit $\eta_M$ is the $H$-linear map
\begin{equation}\label{unit}
\eta_M:M\to UFM=(\XX \otimes_H M)^I,\quad
m\mapsto 1_I\otimes m.
\end{equation}

\begin{lem}\label{unit_proj}
For any projective $H$-module $P$ the unit $\eta_P$ is an isomorphism.
\end{lem}

\begin{proof}
It is immediate that $\eta_H$ is an isomorphism. Since $U$ and $F$ are additive functors and preserve arbitrary direct sums, $\eta_P$ is an isomorphism for any projective $H$-module $P$. 
\end{proof}

As seen in Example \ref{standard_proj_class} we have the standard projective class $(\mathcal{P}',\mathcal{E}')$ on $\Mod(H)$ where $\mathcal{P}'$ is the class of all projective $H$-modules and $\mathcal{E}'$ is the class of all surjective $H$-linear maps.

\begin{defin}\label{I_proj} We define the \emph{$I$-projective class} on $\Rep_k^\infty(G)$ to be the pullback $(\mathcal{P},\mathcal{E})$ of the standard projective class $(\mathcal{P}',\mathcal{E}')$ along $U$. Explicitly,
\begin{itemize}
\item[(i)] $V\in\mathcal{P}$ if and only if there is a split surjection $FP\to V$ for some projective $H$-module $P$ and
\item[(ii)] $f:V\to W$ is in $\mathcal{E}$ if and only if $Uf:UV\to UW$ is surjective.
\end{itemize}
We say that $V\in \Rep_k^\infty(G)$ is \emph{$I$-projective} if $V\in \mathcal{P}$ and that a $G$-equivariant map $f:V\to W$ is an \emph{$I$-epimorphism} if $f\in \mathcal{E}$. Furthermore, we say that $V\in \Rep_k^\infty(G)$ is \emph{$I$-free} if $V\cong FY$ for some free $H$-module $Y$. These are precisely the $G$-representations isomorphic to direct sums of copies of $\XX $. A sequence $V_1\to V_2\to V_3$ in $\Rep_k^\infty(G)$ is called \emph{$I$-exact} if it is $\mathcal{P}$-exact, i.e.\ if it is a complex and $UV_1\to UV_2\to UV_3$ is an exact sequence of $H$-modules (cf.\ Lemma \ref{pullback}). In this context $\mathcal{P}$-resolutions will also be called {\it $I$-resolutions}. Finally, we say that $V\in\Rep^\infty_k(G)$ is {\it generated by its $I$-invariants} if $V=k[G]\cdot V^I=k[G]\cdot UV$ inside $V$. Since $k[G]\cdot UV=\im(\varepsilon_V)$ this is true if and only if the counit $\varepsilon_V:FUV\to V$ is surjective.
\end{defin}

\begin{lem}\label{I_exact}
Let $V_1\stackrel{f}{\to}V_2\stackrel{g}{\to}V_3$ be an $I$-exact sequence in $\Rep^\infty_k(G)$. Then $\im(Uf)=U\im(f)=U\ker(g)=\ker(Ug)$.
\end{lem}

\begin{proof}
Since $\im(f)\subseteq\ker(g)$ and the right adjoint $U$ commutes with kernels the $I$-exactness gives $\ker(Ug) = \im(Uf)\subseteq U\im(f)\subseteq U\ker(g)=\ker(Ug)$.
\end{proof}

We use the unit $\eta_H$ to identify $H\overset{\cong}{\longrightarrow}UFH\cong\XX ^I$ as $(H,H)$-bimodules. Explicitly, this map is given by $h\mapsto 1_I \cdot h$. The terminology in Definition \ref{I_proj} is justified by the following result.

\begin{lem}\label{summand_I_proj} For $V\in \Rep_k^\infty(G)$ the following are equivalent.
\begin{enumerate}
\item[(i)] $V$ is $I$-projective.
\item[(ii)] $V$ is a direct summand of an $I$-free $G$-representation.
\item[(iii)] $UV$ is a projective $H$-module and $V$ is generated by its $I$-invariants.
\end{enumerate}
Under these conditions the counit $\varepsilon_V:FUV\to V$ is an isomorphism.
\end{lem}

\begin{proof}
If $V$ is $I$-projective then by definition $V$ is a direct summand of $FP$ for some projective $H$-module $P$. Since $P$ is a direct summand of a free $H$-module it follows that $FP$ and hence $V$ is a direct summand of an $I$-free representation. That (ii) implies (iii) follows from Lemma \ref{unit_proj} and because the property of being generated by $I$-invariants is stable under quotients.\\

Finally, if $UV$ is projective then $FUV$ is $I$-projective by definition. By Lemma \ref{unit_proj} the unit $\eta_{UV}$ is an isomorphism. It always has the splitting $U\varepsilon_V$. Thus, we see that the counit $FUV\to V$ induces an isomorphism on $I$-invariants. Since the functor $U$ reflects monomorphisms it follows that $FUV\to V$ is injective with image equal to the subrepresentation of $V$ generated by $UV$. This shows that (iii) implies (i) and that in this case the counit is an isomorphism.
\end{proof}

\begin{cor}\label{proj_equiv} The functors $F$ and $U$ induce inverse equivalences of categories between the full subcategories of projective $H$-modules and of $I$-projective $G$-representations, respectively.
\end{cor}

\begin{proof}
Let $V\in\Rep_k^\infty(G)$ be $I$-projective and let $P\in\Mod(H)$ be projective. By definition $FP$ is $I$-projective. Moreover, $UV=V^I$ is projective by Lemma \ref{summand_I_proj}. Thus, the two functors give a well-defined adjunction between the two full subcategories. Since the unit $\eta_P$ and the counit $\varepsilon_V$ are isomorphisms by Lemma \ref{unit_proj} and Lemma \ref{summand_I_proj} the statement follows.
\end{proof}

Given $V\in\Rep^\infty_k(G)$, Lemma \ref{pullback} guarantees the existence of an $I$-epi\-mor\-phism $W\to V$ where $W$ is $I$-projective. This can be constructed in a both explicit and functorial way by considering the $H$-linear surjection $H\otimes_kUV\to UV$. Applying the functor $F=\XX \otimes_H(\cdot)$ and composing with the counit $\varepsilon_V$ gives the functorial $G$-equivariant map $\XX \otimes_kUV\to V$. Here the left hand side is even $I$-free. Moreover, on $I$-invariants we get back the surjection $H\otimes_kUV\to UV$ we started with. Thus, the map $\XX \otimes_kUV\to V$ is an $I$-epimorphism. Its image is the $G$-subrepresentation of $V$ generated by $UV=V^I$. Iterating this process with the kernel of $\XX \otimes_kUV\to V$ we see that any representation admits a functorial $I$-resolution.\\  

Let $(\mathcal{P}, \mathcal{E})$ be a projective class on an abelian category $\A$. As in \cite{ChHov}, Definition 2.1, there is a candidate for an associated model structure on the category $\Ch(\A)$ of unbounded chain complexes over $\A$. If this model structure is well-defined we will call it the \emph{$\mathcal{P}$-projective model structure}. For the $I$-projective class on $\Rep^\infty_k(G)$ the definition is as follows. For ease of notation we set $\Ch(G)=\Ch(\Rep^\infty_k(G))$ and write $\Hom_G(V_\bullet,W_\bullet)$ for the set of morphisms in this category.

\begin{defin} \label{I_proj_model_structure}A map $f: V_\bullet\to W_\bullet$ in $\Ch(G)$ is an \emph{$I$-equivalence} if the induced map $\Hom_G(P, V_\bullet)\to\Hom_G(P, W_\bullet)$ is a quasi-isomorphism of complexes of abelian groups for any $I$-projective $P$. Similarly, $f$ is an \emph{$I$-fibration} if $\Hom_G(P, V_\bullet)\to\Hom_G(P, W_\bullet)$ is an epimorphism of complexes of abelian groups for any $I$-projective $P$. Finally, $f$ is an \emph{$I$-cofibration} if $f$ has the left lifting property with respect to all $I$-trivial fibrations, i.e.\ with respect to all maps that are both $I$-fibrations and $I$-equivalences.
\end{defin}

Note that the adjunction $F:\Mod(H)\rightleftarrows\Rep^\infty_k(G):U$ extends to an adjunction $F:\Ch(H)\rightleftarrows\Ch(G):U$ between the respective categories of complexes by working termwise.

\begin{lem}\label{proj_is_transfer}
A map $f: V_\bullet\to W_\bullet$ in $\Ch(G)$ is an $I$-equivalence (resp.\ an $I$-fibration) if and only if $Uf: UV_\bullet\to UW_\bullet$ is a quasi-isomorphism (resp.\ an epimorphism).
\end{lem}

\begin{proof}
Since $U\cong\Hom_G(\XX ,\cdot)$ and since $\XX $ is $I$-projective it follows that $Uf$ is a quasi-isomorphism for every $I$-equivalence $f$. Conversely, for any indexing set $\mathcal{J}$ we have an isomorphism of functors $\Hom_G(\XX ^{\oplus \mathcal{J}},\cdot)\cong \prod_{\mathcal{J}}\Hom_G(\XX ,\cdot)$ and taking homology of complexes commutes with direct products. It follows that if $Uf=\Hom_G(\XX ,f)$ is a quasi-isomorphism (resp.\ an epimorphism) then so is $\Hom_G(V,f)$ for any $I$-free representation $V$. By taking direct summands the same holds for any $I$-projective $V$ (cf.\ Lemma \ref{summand_I_proj} (ii)).
\end{proof}

Our first aim is to show that Definition \ref{I_proj_model_structure} indeed gives a model structure on $\Ch(G)$ whose weak equivalences (resp.\ cofibrations, resp.\ fibrations) are the $I$-equivalences (resp.\ $I$-cofibrations, resp.\ $I$-fibrations). To this end we apply the following result of Christensen and Hovey (cf.\ \cite{ChHov}, Theorem 5.1).

\begin{thm}[Christensen-Hovey]\label{Hovey_Christensen} Let $\A$ be a bicomplete abelian category all of whose objects are small and let $(\mathcal{P},\mathcal{E})$ be a projective class on $\A$. Suppose that $\mathcal{P}$ is determined by a set, i.e.\ that there is a set $\mathcal{S}$ of objects of $\mathcal{P}$ such that $\mathcal{E}$ is precisely the class of $\mathcal{S}$-epimorphisms. Then the $\mathcal{P}$-projective model structure on $\Ch(\A)$ is well-defined and cofibrantly generated.\hfill$\Box$
\end{thm}

As an application we obtain the following result.

\begin{prop}\label{model_Gchain}
There is a cofibrantly generated model structure on $\Ch(G)$ where the fibrations are the chain maps with surjective $I$-invariants and the weak equivalences are the chain maps whose $I$-invariants are quasi-isomorphisms.
\end{prop}

\begin{proof}
In view of Lemma \ref{proj_is_transfer} we just need to check that the hypotheses in Theorem \ref{Hovey_Christensen} are satisfied. Note that $\Rep_k^\infty(G)$ is a Grothendieck category by \cite{Sch}, Lemma 1, and so all of its objects are small by \cite{Hov1}, Proposition A.2. Since $\mathcal{E}$ is precisely the class of morphisms which are $\XX $-epimorphisms the projective class $(\mathcal{P}, \mathcal{E})$ is determined by the singleton $\{\XX \}$.
\end{proof}

We will call the above model structure on $\Ch(G)$ the \emph{$I$-projective model structure}. Its existence can also be deduced from the general results in the previous section. From Lemma \ref{proj_is_transfer} we see that the $I$-projective model structure on $\Ch(G)$ is the right transfer of the projective model structure on $\Ch(H)$. Instead of Theorem \ref{Hovey_Christensen} we could thus have referred to Theorem \ref{right_transfer} and Proposition \ref{path_objects} by constructing suitable path objects in $\Ch(G)$.

\begin{prop}\label{Quillen_equivalence}
The adjoint functors $F:\Ch(H)\rightleftarrows\Ch(G):U$ given by $FM_\bullet=\XX \otimes_H M_\bullet$ and $UV_\bullet=V_\bullet^I$ form a Quillen equivalence with respect to the $I$-projective model structure on $\Ch(G)$ and the projective model structure on $\Ch(H)$. In particular, the derived adjunction
\[
LF:D(H)\rightleftarrows \Ho(\Ch(G)):RU
\]
is an equivalence between the unbounded derived category of $\Mod(H)$ and the homotopy category of the $I$-projective model structure on $\Ch(G)$.
\end{prop}

\begin{proof}
As we observed above, the $I$-projective model structure on $\Ch(G)$ is the right transfer of the projective model structure on $\Ch(H)$. By Lemma \ref{unit_equiv} it suffices to see that the unit $\eta_{X_\bullet}$ is a weak equivalence at all cofibrant (i.e.\ dg-projective) $X_\bullet\in \Ch(H)$. Since $X_\bullet$ is cofibrant it is termwise projective (cf.\ \cite{HovBook}, Lemma 2.3.6). Thus, it follows from Lemma \ref{unit_proj} that $\eta_{X_\bullet}$ is an isomorphism and hence a weak equivalence as required. For the final statement see \cite{HovBook}, Proposition 1.3.13.
\end{proof}

In order to compute the left derived functor $LF$ in Proposition \ref{Quillen_equivalence} let $M_\bullet\in\Ch(H)$ and let $Q_\bullet$ be any dg-projective complex of $H$-modules which is quasi-isomorphic to $M_\bullet$. Then $LFM_\bullet\cong \XX \otimes_HQ_\bullet$ in $\Ho(\Ch(G))$ for the $I$-projective model structure on $\Ch(G)$. To compute the right derived functor $RU$ let $V_\bullet\in \Ch(G)$. Since every object of $\Ch(G)$ is fibrant for the $I$-projective model structure we have $RUV_\bullet\cong UV_\bullet$ in $D(H)$.\\

The results in Proposition \ref{Quillen_equivalence} admit important improvements. To see this we start with the following characterization of the $I$-cofibrant objects in $\Ch(G)$. Using \cite{ChHov}, Proposition 2.5, this gives a description of the $I$-cofibrations in general. Note that even if a Quillen adjunction \mbox{$F:\C\rightleftarrows\D:U$} is a Quillen equivalence the right adjoint $U$ will generally not preserve cofibrant objects. The adjunction $\id:\Ch(S)^\Proj\rightleftarrows\Ch(S)^\Inj:\id$ gives an easy example. However, our situation is much more special.

\begin{lem}\label{I_cofibrations}
For an object $V_\bullet$ of $\Ch(G)$ the following are equivalent.
\begin{itemize}
\item[(i)]$V_\bullet$ is $I$-cofibrant.
\item[(ii)]$V_\bullet$ is termwise generated by its $I$-invariants and the complex $UV_\bullet$ in $\Ch(H)$ is cofibrant in the projective model structure, i.e.\ dg-projective. 
\end{itemize}
If these conditions are satisfied then the counit $\varepsilon_{V_\bullet}:FUV_\bullet\to V_\bullet$ is an isomorphism in $\Ch(G)$.
\end{lem}
\begin{proof}
Let us put $Y=V_\bullet$ and denote by $q_{UY}:QUY\to UY$ the cofibrant replacement of $UY=V_\bullet^I$. The adjoint morphism $\varepsilon_Y Fq_{UY}:FQUY\to Y$ is an $I$-equivalence by Proposition \ref{Quillen_equivalence}. We claim it is a trivial $I$-fibration. To see this we need to check that we get a trivial fibration in $\Ch(H)$ if we apply the functor $U$. Consider the commutative diagram
\[
\xymatrix@C=1.5cm{
UFQUY\ar[r]^{UFq_{UY}}&UFUY\ar[r]^{U\varepsilon_Y}&UY.\\
QUY\ar[u]^{\eta_{QUY}}\ar[r]_{q_{UY}}&UY\ar[u]_{\eta_{UY}}\ar[ur]_{\id_{UY}}&
}
\]
As seen in the proof of Proposition \ref{Quillen_equivalence} the left vertical arrow is an isomorphism because $QUY$ is cofibrant. Since $q_{UY}$ is a trivial fibration, the claim follows.\\

Assuming that $Y$ is $I$-cofibrant it follows that the map $FQUY\to Y$ splits. Thus, $Y$ is a direct summand of $FQUY$ and $UY$ is a direct summand of $UFQUY\cong QUY$. This implies that $UY$ is a cofibrant object of $\Ch(H)$, i.e.\ is dg-projective. Moreover, Lemma \ref{summand_I_proj} shows that $FQUY$ is termwise $I$-projective. The same is then true of its direct summand $Y$. In particular, $Y$ is termwise generated by its $I$-invariants.\\

Conversely, assume that $Y$ satisfies the conditions in (ii). Then $UY$ is dg-projective and hence is termwise projective. It follows from Lemma \ref{summand_I_proj} that $Y$ is termwise $I$-projective and that the counit $\varepsilon_Y:FUY\to Y$ is an isomorphism. Since $FUY$ is $I$-cofibrant by definition of a left Quillen functor it follows that $Y\cong FUY$ is $I$-cofibrant.
\end{proof}

We obtain the following significant strengthening of Proposition \ref{Quillen_equivalence}. It says that the functors $F$ and $U$ restrict to equivalences between the cofibrant-fibrant objects even before dividing out the homotopy relation. Note that in the situation of Proposition \ref{Quillen_equivalence} all objects are fibrant. We may therefore work with the respective classes of cofibrant objects.

\begin{cor}\label{fib_cofib_equivalence}
Endow $\Ch(H)$ with the projective model structure and $\Ch(G)$ with the $I$-projective model structure of Proposition \ref{model_Gchain}. The functors $U=(\cdot)^I$ and $F=\XX \otimes_H(\cdot)$ restrict to  inverse equivalences of categories $\Ch(G)_c\rightleftarrows\Ch(H)_c$.
\end{cor}
\begin{proof}
If $Y\in\Ch(G)$ is $I$-cofibrant then $UY$ is cofibrant in $\Ch(H)$ by Lemma \ref{I_cofibrations}. If $Z\in\Ch(H)$ is cofibrant then $FZ$ is $I$-cofibrant by definition of a left Quillen functor. The counit $\varepsilon_Y:FUY\to Y$ is an isomorphism by Lemma \ref{I_cofibrations} and the unit $\eta_Z:Z\to UFZ$ is an isomorphism by the proof of Proposition \ref{Quillen_equivalence}.
\end{proof}

For complexes concentrated in degree zero this reflects the fact that $F$ and $U$ restrict to inverse equivalences between the category of projective $H$-modules and the category of $I$-projective $G$-representations, respectively (cf.\ Corollary \ref{proj_equiv}). Given $V,W\in\Rep^\infty_k(G)$ we denote by $\Ext^n_{G,I}(V,W)$ the relative extension groups defined in (\ref{relative_Ext}) formed with respect to the $I$-projective class in Definition \ref{I_proj}. The following result is a formal consequence of Proposition \ref{Quillen_equivalence}.

\begin{cor}\label{Ext_comparison}
If $V,W\in\Rep_k^\infty(G)$ then there are functorial bijections $\Ext^n_{G,I}(V,W)\cong\Ext_H^n(UV,UW)$ for all $n\geq 0$.
\end{cor}
\begin{proof}
Denote by $\mathcal{D}=\Ho(\Ch(G))$ the homotopy category of $\Ch(G)$ endowed with the $I$-projective model structure. By Proposition \ref{Quillen_equivalence} and \cite{ChHov}, Corollary 2.14, we have
\begin{eqnarray*}
\Ext^n_{G,I}(V,W)&\cong&\Hom_\mathcal{D}(V,W[-n])\cong\Hom_{D(H)}(UV,UW[-n])\\
&\cong&\Ext^n_H(UV,UW)
\end{eqnarray*}
functorially in $V$ and $W$ for all $n\geq 0$.
\end{proof}

The existence of the $I$-cofibrant replacement shows that if $V_\bullet\in\Ch(G)$ then there is a trivial $I$-fibration $f:Q_\bullet\to V_\bullet$ where $Q_\bullet$ is $I$-cofibrant. By Lemma \ref{proj_is_transfer} and Lemma \ref{I_cofibrations} this means that $Uf:UQ_\bullet\to U V_\bullet$ is a surjective quasi-isomorphism where $UQ_\bullet \in\Ch(H)$ is dg-projective. It will be useful to have the following slight variant of this fact generalizing the existence of enough $I$-projectives in $\Rep^\infty_k(G)$ discussed earlier.
\begin{lem}\label{enough_I_projectives}
For any $V_\bullet\in\Ch(G)$ there is an $I$-fibration $f:Q_\bullet\to V_\bullet$ such that $UQ_\bullet\in\Ch(H)$ is a projective complex.
\end{lem}

\begin{proof}
By Lemma \ref{complexes_Grothendieck} (ii) there is a surjection $g:P_\bullet\to UV_\bullet$ in $\Ch(H)$ where $P_\bullet\in\Ch(H)$ is projective. Setting $Q_\bullet=FP_\bullet\in\Ch(G)$ the unit of the adjunction $\eta_{P_\bullet}:P_\bullet\to UQ_\bullet=UFP_\bullet$ is an isomorphism by Lemma \ref{unit_proj}. Therefore, $UQ_\bullet\in\Ch(H)$ is projective. Moreover, the map
\[
f:Q_\bullet=FP_\bullet\stackrel{Fg}{\longrightarrow}FUV_\bullet\stackrel{\varepsilon_{V_\bullet}}{\longrightarrow}V_\bullet
\]
is an $I$-fibration by Lemma \ref{proj_is_transfer} because $Uf\eta_{P_\bullet}=U\varepsilon_{V_\bullet}UFg\eta_{P_\bullet}=g$ is surjective.
\end{proof}

\begin{rem} The above result can also be deduced from the axioms of a model category directly. Indeed, for any $V_\bullet\in\Ch(G)$ we can factorize the map $0\to V_\bullet$ into a trivial $I$-cofibration followed by an $I$-fibration, i.e.\ there is an $I$-fibration $f:Q_\bullet\to V_\bullet$ where $Q_\bullet\in\Ch(G)$ is trivially cofibrant. By Lemma \ref{I_cofibrations} this implies that $UQ_\bullet\in\Ch(H)$ is both cofibrant (i.e.\ dg-projective) and trivial (i.e.\ exact) hence is projective.
\end{rem}

For the sake of completeness we note that it is also possible to start with a model structure on $\Ch(G)$ and apply a left transfer construction to get a model structure on $\Ch(H)$. However, in general it does not make sense to talk about a projective model structure on $\Ch(G)$ unless $G$ is discrete.

\begin{rem}\label{not_enough_projectives}
Let $G$ be a non-discrete profinite group admitting an open subgroup which is pro-$p$. If $k$ is a field of characteristic $p$ then the only projective object of $\Rep_k^\infty(G)$ is the zero object.
\end{rem}
\begin{proof}
Let $U''\subsetneqq U'$ be open pro-$p$ subgroups of $G$ such that $U''$ is normal in $U'$. Note first that the space of $U'$-invariants of $\ind_{U''}^{U'}(k)=k[U'/U'']$ is spanned by $\sum_{u\in U'/U''}u$. This maps to zero in $\ind_{U''}^{U'}(k)_{U'}$ because $p$ divides $(U':U'')$ and $k$ has characteristic $p$. It follows that the canonical map $\ind_{U''}^{U'}(W)^{U'}\to\ind_{U''}^{U'}(W)_{U'}$ is zero for any trivial $U''$-representation $W$.\\

Now assume that $P\in\Rep_k^\infty(G)$ is a projective object and that $U'$ is an open pro-$p$ subgroup of $G$. The restriction functor $\Rep_k^\infty(G)\to\Rep_k^\infty(U')$ preserves projective objects because it admits the exact right adjoint $\Ind_{U'}^G=\ind_{U'}^G$. Since $G$ is non-discrete there is an open normal subgroup $U''\subsetneqq U'$. Consider the surjection $\ind_{U''}^{U'}(P_{U'})\to P_{U'}$. Since $P$ is projective the canonical map $P\to P_{U'}$ lifts to a map $P\to \ind_{U''}^{U'}(P_{U'})$. Therefore, the canonical map $P^{U'}\to P_{U'}$ factors through $\ind_{U''}^{U'}(P_{U'})^{U'}\to\ind_{U''}^{U'}(P_{U'})_{U'}\to P_{U'}$, hence is zero as seen above. This also applies to $U''\subseteq {U'}$ so that the map $P^{U''}\to P_{U''}\to P_{U'}$ is zero. Letting $U''\subseteq {U'}$ vary, we obtain that the canonical map $P=\bigcup_{U''}P^{U''}\to P_{U'}$ is zero.\\

Now fix an open pro-$p$ subgroup $U$ of $G$. Letting $U''$ run through the open normal subgroups of $U$ the canonical surjection $V=\bigoplus_{U''}P^{U''}\to P$ splits. For any open normal subgroup $U'\subseteq U$ and any $W\in\Rep_k^\infty(U)$ we let $W(U')=\ker(W\to W_{U'})$. Since $U'$ acts trivially on $P^{U'}$ we have $(P^{U'})(U')=0$. Thus, $V(U')=\bigoplus_{U''}(P^{U''})(U')\subseteq\bigoplus_{U''\neq U'}P^{U''}$. As seen above we have $P=P(U')$ for any open normal subgroup $U'$ of $U$ and hence $P\subseteq\bigcap_{U'} V(U')=0$.
\end{proof}

Instead we consider the injective model structure on $\Ch(G)$ introduced in \S\ref{section_1} and show that it admits a left transfer to $\Ch(H)$ in the following sense.

\begin{prop}\label{Hecke_left_transfer} There is a model structure on $\Ch(H)$ where the cofibrations are the maps $g$ such that $Fg$ is a monomorphism and the weak equivalences are the maps $g$ such that $Fg$ is a quasi-isomorphism. The fibrations are the maps satisfying the right lifting property with respect to all trivial cofibrations.
\end{prop}

In analogy to the right transfer the left transfer has an apparent ambiguity concerning the notion of a trivial fibration. This could either be a morphism satisfying the right lifting property with respect to all cofibrations or it could be a fibration which is also a weak equivalence. We need to see that these two classes of morphisms coincide. To distinguish them we call a map $f$ in $\Ch(H)$ {\it coanodyne} if it satisfies the right lifting property with respect to all cofibrations in the sense of Proposition \ref{Hecke_left_transfer}.

\begin{rem}\label{coanodyne_spli_epi} Note that any coanodyne map $f: X\to Y$ in $\Ch(H)$ is a split epimorphism. Indeed, the map $0\to Y$ is a cofibration for the left transfer. By applying the right lifting property to
\[
\begin{tikzcd}
0\arrow{r}\arrow{d} & X\arrow{d}{f}\\
Y\arrow[swap]{r}{\id} & Y
\end{tikzcd}
\]
we obtain the desired splitting.
\end{rem}

If the left transfer is a model structure then the coanodyne morphisms are precisely the fibrations which are also weak equivalences. Moreover, one needs to show that functorial factorizations exist. In general this involves serious set-theoretic issues. However the following result shows that under certain assumptions on the categories these problems are easy to solve (cf.\ \cite{BHKKRS}, Theorem 2.23).

\begin{thm}\label{left_transfer_exists} Suppose we have an adjunction $F:\mathcal{C}\rightleftarrows\mathcal{D}:U$ between locally presentable categories where $\mathcal{D}$ carries a cofibrantly generated model structure. Then the left transfer to $\mathcal{C}$ exists if and only if $Ff$ is a weak equivalence in $\mathcal{D}$ for every coanodyne morphism $f$ in $\mathcal{C}$. \hfill$\Box$
\end{thm}

We note that together with $\Mod(H)$ and $\Rep^\infty_k(G)$ also the categories $\Ch(H)$ and $\Ch(G)$ are Grothendieck categories (cf.\ Lemma \ref{complexes_Grothendieck}) and hence are locally presentable (cf.\ \cite{Beke}, Proposition 3.10).

\begin{prop}\label{coanodyne_is_qi} If $f_\bullet: P_\bullet\to Q_\bullet$ is a coanodyne map in $\Ch(H)$ then the complexes $K_\bullet=\ker(f_\bullet)$ and $FK_\bullet =\XX\otimes_H K_\bullet$ are contractible.
\end{prop}

\begin{proof}
Let $T_\bullet$ be the mapping cone of $\id_{K_\bullet}$ defined by $T_n=K_n\oplus K_{n-1}$ with differential $d(x,y)=(dx+y, -dy)$. Consider the canonical map $\iota_\bullet:K_\bullet\to T_\bullet$ given by $\iota (x)=(x,0)$. Since this is a termwise split injection so is $F\iota_\bullet$. In particular, $F\iota_\bullet$ is a cofibration and so is $\iota_\bullet$ by definition of the left transfer. Since $f_\bullet$ is coanodyne we may apply the right lifting property to the square
\[
\begin{tikzcd}
K_\bullet\arrow[hook]{r}\arrow[swap]{d}{\iota_\bullet} & P_\bullet\arrow{d}{f_\bullet}\\
T_\bullet\arrow[swap]{r}{0}\arrow[dashrightarrow]{ur}{g_\bullet} & Q_\bullet
\end{tikzcd}
\]
and obtain a map $g_\bullet:T_\bullet\to P_\bullet$ making the two triangles commute. This implies $f_\bullet g_\bullet=0$ whence $g_\bullet$ has image in $K_\bullet$. We may therefore view $g_\bullet$ as a map $T_\bullet\to K_\bullet$ with the property that $g_\bullet \iota_\bullet=\id_{K_\bullet}$. For any $n\in\Z$ this allows us to write $g_n(x,y)=x+s_ny$ for some $H$-linear map $s_n:K_{n-1}\to K_n$. Let $x\in K_{n+1}$ and $y\in K_n$. Since $g_\bullet$ is a chain map we have
\begin{align*}
d_{n+1}x+y-s_nd_ny&=g_n(d_{n+1}x+y, -d_ny)\\
&=g_nd_{n+1}(x,y)\\
&=d_ng_{n+1}(x,y)\\
&=d_{n+1}x+d_ns_{n+1}y
\end{align*}
and hence $y=s_nd_ny+d_ns_{n+1}y$. Thus, $s_\bullet$ defines a chain homotopy between the zero map and the identity map on $K_\bullet$. By functoriality $F s_\bullet$ is then a chain homotopy between the zero map and the identity on $FK_\bullet$.
\end{proof}

\begin{rem} We note that the proof of Proposition \ref{coanodyne_is_qi} implies that coanodyne maps in $\Ch(H)$ are quasi-isomorphisms.
\end{rem}

\begin{proof}[Proof of Proposition \ref{Hecke_left_transfer}]
By Theorem \ref{left_transfer_exists} we just need to check that if $f_\bullet: P_\bullet\to Q_\bullet$ is a coanodyne map in $\Ch(H)$ then $Ff_\bullet$ is a quasi-isomorphism in $\Ch(G)$. To see this let $K_\bullet=\ker(f_\bullet)$. By Remark \ref{coanodyne_spli_epi} the map $f_\bullet$ is a split epimorphism. Therefore, also $Ff_\bullet$ is a split epimorphism and we have $\ker(Ff_\bullet)=F\ker(f_\bullet)=FK_\bullet$. Since $FK_\bullet$ is acyclic (cf.\ Proposition \ref{coanodyne_is_qi}) the long exact homology sequence shows that $Ff_\bullet$ is a quasi-isomorphism.
\end{proof}

The model structure on $\Ch(H)$ constructed in Proposition \ref{Hecke_left_transfer} will be called the {\it $G$-injective model structure}. It makes $F:\Ch(H)\rightleftarrows\Ch(G):U$ a Quillen adjunction. We hope to give an in-depth study of this adjunction elsewhere.


\section{Gorenstein projective model structures}\label{section_3}

%
%


Let $\mathcal{A}$ be an abelian category with enough projectives. An object $A\in\mathcal{A}$ is called {\it Gorenstein projective} if there is an acyclic complex
\[
Y = \cdots\longrightarrow Y_1\stackrel{d_1}{\longrightarrow}Y_0\stackrel{d_0}{\longrightarrow} Y_{-1}\longrightarrow\cdots
\]
of projective objects of $\mathcal{A}$ for which $A=Z_0Y=B_0Y$ and which remains exact upon applying $\Hom_{\mathcal{A}}(\cdot,P)$ for any projective object $P$ of $\mathcal{A}$.\\

A ring $S$ is called {\it Gorenstein} if it is both left and right noetherian and if it has finite injective dimension both as a left and as a right $S$-module. In this case the left and right selfinjective dimensions of $S$ coincide (cf.\ \cite{EJ11}, Proposition 9.1.8). If $n$ denotes their common value then $S$ is called an {\it $n$-Gorenstein ring}.\\

If $S$ is an $n$-Gorenstein ring then an $S$-module $M$ has finite projective dimension if and only if it has finite injective dimension. In this case both dimensions are bounded above by $n$ (cf.\ \cite{EJ11}, Theorem 9.1.10). Moreover, $M$ is Gorenstein projective as an object of $\Mod(S)$ if and only if $\Ext^i_S(M,P)=0$ for all $i\geq 1$ and all projective $S$-modules $P$  (cf.\ \cite{EJ11}, Corollary 11.5.3). This is true if and only if $\Ext^i_S(M,P)=0$ for all $i\geq 1$ and all modules $P$ of finite projective dimension.

\begin{rem}\label{complete_resolution}
If $S$ is an $n$-Gorenstein ring then an $S$-module $M$ is Gorenstein projective if and only if $M=Z_0Y$ for an acyclic complex $Y\in\Ch(S)$ of projective $S$-modules. Indeed, for any projective $S$-module $P$ we then have $\Ext^i_S(M,P)= \Ext^i_S(Z_0Y,P) = \Ext^{i+n}_S(Z_{-n}Y,P)=0$ for any $i>0$ because $P$ has injective dimension at most $n$. The same argument shows that all cycles of $Y$ are Gorenstein projective. That the complex $Y$ remains acyclic upon applying $\Hom_S(\cdot, P)$ is then automatic. To see this one simply applies $\Hom_S(\cdot, P)$ to the short exact sequences $0\to Z_jY\to Y_j\to Z_{j-1}Y\to 0$ for any $j\in \Z$.
\end{rem}

The following fundamental theorem is due to Hovey (cf.\ \cite{Hov2}, Theorem 8.6).
\begin{thm}[Hovey]\label{GP_R_Mod}

Let $S$ be a Gorenstein ring. On $\Mod(S)$ there is a cofibrantly generated model structure for which
\begin{itemize}
\item the cofibrations are the monomorphisms with Gorenstein projective cokernel,
\item the fibrations are the epimorphisms,
\item the trivial objects are the modules of finite projective dimension.\hfill$\Box$
\end{itemize}
\end{thm}

We call this the {\it Gorenstein projective model structure} on $\Mod(S)$. The class $\C$ (resp.\ $\F$, resp.\ $\T$) of cofibrant (resp.\ fibrant, resp.\ trivial) objects is that of Gorenstein projective modules (resp.\ all modules, resp.\ modules of finite projective dimension). By \cite{Hov2}, Corollary 8.5, the class $\C\cap\T$ is that of projective $S$-modules. The subcategory $\T$ of $\Mod(S)$ is thick and both $(\C\cap\T,\F)$ and $(\C,\F\cap\T)$ are functorially complete cotorsion pairs. Letting $\mathcal{P}$ denote the class of all short exact sequences in $\Mod(S)$ the Gorenstein projective model structure on $\Mod(S)$ is obtained from these data as in Theorem \ref{cotorsion_pairs} (ii). Since all objects are fibrant and since the class of Gorenstein projective modules is closed under kernels of epimorphisms the Gorenstein projective model structure on $\Mod(S)$ is abelian and hereditary. In particular, it is stable which is also proven directly in \cite{Hov2}, Theorem 9.3.\\

Assume that $S$ is Gorenstein. Given an $S$-module $M$ the cofibrant replacement functor $Q$ for the Gorenstein projective model structure on $\Mod(S)$ gives a functorial exact sequence
\[
0\longrightarrow KM\longrightarrow QM\longrightarrow M\longrightarrow 0
\]
in which $QM$ is Gorenstein projective and $KM$ has finite projective dimension. Let us recall from \cite{EJ11}, Theorem 11.5.1, or \cite{Zha08}, Theorem 3.5, how to construct such an exact sequence at least non-functorially.\\

Denote by $n$ the selfinjective dimension of $S$. One first chooses a projective resolution $Q_\bullet\to M\to 0$ of $M$ and sets $G=\coker(Q_{n+1}\to Q_n)$. Then $G$ is Gorenstein projective by \cite{EJ11}, Theorem 10.2.14. By definition, $G$ admits a complete projective resolution $P_\bullet$ as above. By shifting we see that also $N=\im(P_{-n}\to P_{-n-1})=\ker(P_{-n-1}\to P_{-n-2})$ is Gorenstein projective. Moreover, the exact sequence $0\to G\to P_{-1}\to\cdots\to P_{-n}\to N\to 0$ is $\Hom_S(\cdot,Q)$-acyclic for any projective $S$-module $Q$. Therefore, we can inductively choose $S$-linear maps $P_i\to Q_{n+i}$ to obtain a commutative diagram
\[
\xymatrix{
0\ar[r]&G\ar[r]\ar@{=}[d]&P_{-1}\ar[r]\ar[d]&\cdots\ar[r]&P_{-n}\ar[r]\ar[d]&N\ar[r]\ar[d]&0\\
0\ar[r]&G\ar[r]&Q_{n-1}\ar[r]&\cdots\ar[r]&Q_0\ar[r]&M\ar[r]&0.
}
\]
The mapping cone of the vertical homomorphism yields an exact sequence $0\to P_{-1}\to P_{-2}\oplus Q_{n-1}\to \cdots \to P_{-n}\oplus Q_1\to N\oplus Q_0\to M\to 0$. Here $N\oplus Q_0$ is Gorenstein projective and the kernel of the epimorphism to $M$ has projective dimension at most $n-1$.
\begin{rem}\label{stable_modules}
Denote by $\Proj(S)$ and $\mathrm{GProj}(S)$ the full subcategory of $\Mod(S)$ consisting of all projective and all Gorenstein projective $S$-modules, respectively. The results of \cite{Hov2}, Proposition 9.1 and Proposition 9.2, show that the homotopy category $\Ho(\Mod(S))$ can be viewed as the full subcategory $\mathrm{GProj}(S)/\mathrm{Proj}(S)$ of the stable module category $\Mod(S)/\mathrm{Proj}(S)$ of $S$ consisting of all Gorenstein projective modules.
\end{rem}
We point out that Theorem \ref{GP_R_Mod} admits a variant which is based on the dual notion of a {\it Gorenstein injective} module (cf.\ \cite{Hov2}, Theorem 8.4). The corresponding homotopy category is equivalent to the quotient $\mathrm{GInj}(S)/\mathrm{Inj}(S)$ of the category $\mathrm{GInj}(S)$ of Gorenstein injective modules by the category $\mathrm{Inj}(S)$ of injective $S$-modules (cf.\ \cite{Hov2}, Proposition 9.1 and Proposition 9.2). It was shown by Krause in \cite{Kra05}, Proposition 7.13, that the latter is equivalent to the {\it singularity category} $K_\ac(\mathrm{Inj}(S))$ of $S$, i.e.\ to the category of acyclic complexes of injective $S$-modules up to chain homotopy. From the model categorical point of view this was taken up by Becker in \cite{Beck14} who again treats the projective and the injective situation simultaneously. We point out that results of this form have their origin in the seminal but unpublished work \cite{Buc} of Buchweitz. We refer to the end of \cite{Kra05}, \S7, for a more comprehensive list of historical remarks.

\begin{thm}[Becker]\label{Becker}
\begin{enumerate}[wide]
	\item[(i)]For any ring $S$ there is a cofibrantly generated model structure on $\Ch(S)$ for which 
	\begin{itemize}
		\item the cofibrations are the monomorphisms with acyclic and termwise projective cokernel,
		\item the fibrations are the epimorphisms,
		\item the trivial objects are the complexes $Y$ satisfying $\Ext^1_{\Ch(S)}(X,Y)=0$ for all termwise projective acyclic complexes $X$.
	\end{itemize}
	Its homotopy category is equivalent to the category $K_\ac(\mathrm{Proj}(S))$ of acyclic complexes of projective $S$-modules up to chain homotopy.
	\item[(ii)]If $S$ is a Gorenstein ring then the adjoint functors 
\[
Q_0:\Ch(S)\rightleftarrows\Mod(S):\iota_0
\]
defined in (\ref{left_right_adjoints}) form a Quillen equivalence for the model structure on $\Ch(S)$ as in (i) and the Gorenstein projective model structure on $\Mod(S)$. 
\end{enumerate}
\end{thm}

\begin{proof}
Part (i) is \cite{Beck14}, Proposition 2.2.1 (1), which is applied as in \cite{Beck14}, \S3. The description of the corresponding homotopy category follows from \cite{Beck14}, Proposition 2.2.1 (1) and Example 1.4.7. The main point is that the right homotopy relation between cofibrant objects is the usual chain homotopy of complexes. To see this note that the trivially cofibrant objects are the projective complexes as is shown in the proof of \cite{Beck14}, Proposition 2.2.1 (1).  Therefore, one can argue as in \cite{Hov2}, Proposition 9.1. Finally, part (ii) is \cite{Beck14}, Proposition 3.1.3.
\end{proof}
We call this the {\it singular projective model structure} on $\Ch(S)$. Its existence can also be deduced from the general result \cite{EJ11b}, Theorem 7.2.15, of Enochs and Jenda. The class $\C$ (resp.\ $\F$) of cofibrant (resp.\ fibrant) objects consists of the acyclic and termwise projective (resp.\ all) complexes. The class $\T$ of trivial objects is given by the condition $\T=\C^\perp$ as in Theorem \ref{Becker} (i). The singular projective model structure is obtained from $\C$, $\F$ and $\T$ as in Theorem \ref{cotorsion_pairs} (ii) where we let $\mathcal{P}$ denote the class of all short exact sequences in $\Ch(S)$. Again, this is an abelian hereditary model structure and the associated model category is stable (cf.\ \cite{Beck14}, Corollary 1.1.15). The loop and suspension functors can be made explicit as follows.  

\begin{lem}\label{explicit_loop_susp}
The adjunction $[-1]:\Ch(S)\rightleftarrows\Ch(S):[1]$ is a Quillen equivalence for the singular projective model structure on $\Ch(S)$. The left derived functor of $[-1]$ coincides with the suspension functor $\Sigma$ and the right derived functor of $[1]$ coincides with the loop functor $\Omega$.
\end{lem}

\begin{proof}
There are isomorphisms $\Ext^1_{\Ch(S)}(X[-1],Y)\cong\Ext^1_{\Ch(S)}(X,Y[1])$ and $\Ext^1_{\Ch(S)}(X[1],Y)\cong\Ext^1_{\Ch(S)}(X,Y[-1])$ for all $X,Y\in\Ch(S)$. This implies that both functors preserve the class of trivial objects. By exactness they preserve all weak equivalences. Since they are inverse to each other this implies that they form a Quillen equivalence in both directions.\\

In order to compute the suspension functor let $X\in\Ch(S)$ be cofibrant and factorize $X\to 0$ into a cofibration followed by a trivial fibration. This gives an exact sequence $0\to X\to Y\to Z\to 0$ where $Y$ is trivial and $Z$ is cofibrant. Then $Z$ is the suspension of $X$ in $\Ho(\Ch(S))$. Since $Z$ is termwise projective the above sequence is termwise split. As usual, we obtain a map of complexes $Z\to X[-1]$ which is unique up to chain homotopy. However, together with $X$ also $Y$ is cofibrant since $X\to Y$ is a cofibration. As remarked in the proof of Theorem \ref{Becker} this implies that $Y$ is a projective complex, hence is contractible. It follows that the map $Z\to X[-1]$ is a homotopy equivalence in the usual sense of chain complexes. As remarked earlier this implies that $Z\to X[-1]$ is a weak equivalence because $Z$ and $X[-1]$ are cofibrant. This implies $L[-1]\cong\Sigma$. That $R[1]\cong\Omega$ follows formally from the uniqueness of adjoints.
\end{proof}

Note that if $S$ is a Gorenstein ring then the Gorenstein projective model structure on $\Mod(S)$ is Quillen equivalent to its Gorenstein injective counterpart (cf.\ \cite{Hov2}, page 583). It follows from the results of Becker that the corresponding homotopy category is also equivalent to the category $K_\ac(\mathrm{Inj}(S))$ of complexes of injective $S$-modules up to chain homotopy (cf.\ \cite{Beck14}, Proposition 2.2.1 (2) and Proposition 3.1.5).\\

By a result of Krause, the full subcategory of compact objects of $K_\ac(\mathrm{Inj}(S))$ is equivalent to the quotient of the bounded derived category $D^b(\mathrm{noeth}(S))$ of noetherian $S$-modules by the subcategory of perfect complexes (cf.\ \cite{Kra05}, Corollary 5.4 and \cite{Sta21}, 07LT). Quotients of this form are usually called {\it singularity categories} in algebraic geometry. This is where the name for the above model structure derives from. We would also like to emphasize that the results of \cite{Beck14} and \cite{Kra05} are much more precise in the sense that they are derived from the existence of various {\it recollements} between triangulated categories.


\section{The right transfer to smooth $G$-representations}\label{section_4}

%
%


We continue to denote by $k$ a field of characteristic $p$ and by $G$ a topological group admitting an open subgroup $I$ which is pro-$p$. We wish to apply the results of \S\ref{section_3} to the ring $S=H=\End_G(\XX)^\op$. In the following we will therefore often make the assumption that the ring $H$ is Gorenstein.
\begin{ex}\label{H_Gorenstein}The ring $H$ is Gorenstein in the following situations.
\begin{itemize}[wide]
\item[(i)]Let $G$ be a finite group with a split $BN$-pair of characteristic $p$. Recall that then $B$ admits a normal $p$-Sylow subgroup $U$ such that $B=U\rtimes T$ with $T=B\cap N$. Setting $I=U$ and endowing $G$ with the discrete topology the Hecke algebra $H$ is a Frobenius algebra by a result of Tinberg (cf.\ \cite{Tin}, Proposition 3.7). In this case $H$ is even selfinjective. As an example one may take the group $G=\mathbb{G}(K)$ of rational points of a split connected reductive group $\mathbb{G}$ over a finite field $K$ of characteristic $p$ and $I=U$ the group of rational points of the unipotent radical of a Borel subgroup of $\mathbb{G}$.
\item[(ii)]Let $\mathbb{G}(K)$ denote the group of rational points of a split connected reductive group $\mathbb{G}$ over a nonarchimedean local field $K$ of residue characteristic $p$. The group $\mathbb{G}(K)$ carries a natural locally profinite topology. Denote by $\mathcal{X}$ the semisimple Bruhat-Tits building of $\mathbb{G}(K)$, by $C$ a chamber of $\mathcal{X}$ and by $\sigma$ a facet of $\mathcal{X}$ contained in the closure of $C$. Denote by $P_\sigma$ the parahoric subgroup of $\sigma$ and by $I_C$ the pro-$p$ radical of $P_C$. Both of them are compact open subgroups of $\mathbb{G}(K)$ and we have $I_C\subseteq P_\sigma$. If $G=P_\sigma$ and if $I=I_C$ then the ring $H$ is isomorphic to the Hecke algebra of a finite group with a split $BN$-pair of characteristic $p$ as in (i). Therefore, $H$ is a Frobenius algebra and hence is selfinjective.
\item[(iii)]We keep the notation of (ii). If $P_\sigma^\dagger$ denotes the stabilizer of $\sigma$ in $\mathcal{X}$ then $P_\sigma^\dagger$ is an open subgroup of $\mathbb{G}(K)$ containing $P_\sigma$. If $G=P_\sigma^\dagger$ and if $I=I_C$ then $H$ is a Gorenstein ring by a result of Ollivier and Schneider (cf.\ \cite{OS14}, Proposition 5.5). Its selfinjective dimension is equal to the rank of the center of $\mathbb{G}$.
\item[(iv)]We keep the notation of (ii). If $G=\mathbb{G}(K)$ and if $I=I_C$ then $H$ is a Gorenstein ring by a result of Ollivier and Schneider (cf.\ \cite{OS14}, Theorem 0.1). Its selfinjective dimension is bounded above by the rank of $\mathbb{G}$. 
\end{itemize}
\end{ex}
In all of the above examples the characteristic of $k$ may in fact be arbitrary. However, we will continue to assume that $k$ is of characteristic $p$. In the situation of Example \ref{H_Gorenstein} (iv) the group $I=I_C$ is a called a {\it pro-$p$ Iwahori subgroup} of $G=\mathbb{G}(K)$  and $H$ is called the corresponding {\it pro-$p$ Iwahori-Hecke algebra} over $k$. Since this is the motivating example for our work we decided to stick to this terminology in general.

\begin{rem}\label{homotopy_cat_trivial}
If the Gorenstein ring $H$ has finite global dimension then all $H$-modules have finite projective dimension and $\Ho(\Mod(H))=0$. In the above examples this happens only in exceptional cases. In the situation of Example \ref{H_Gorenstein} (iv) for instance the global dimension of $H$ was studied in \cite{OS14}, \S7. If the semisimple rank of $\mathbb{G}$ is positive and if the residue class field of $K$ is not $\mathbb{F}_2$ then the global dimension of $H$ is infinite (cf.\ \cite{OS14}, Corollary 7.2 and Lemma 7.3). Moreover, the {\it supersingular} $H$-modules tend to have infinite projective dimension (cf.\ \cite{Koz}, \S1, for a more precise result). 
\end{rem}

For the following result recall that we denote by $B_0$ the $0$-th boundary functor introduced in \S\ref{section_1}.

\begin{prop}\label{GP_split_mono} Let $Y=(Y_\bullet,d_\bullet)$ be an acyclic complex of $H$-modules.
\begin{itemize}
\item[(i)]There is a functorial exact sequence ${0\to H_1FY\to FB_0Y\to B_0FY\to 0}$. 
\item[(ii)]If $Y$ is termwise projective then the canonical map $B_0Y\to UB_0FY$ is an $H$-linear bijection and the sequence in (i) is $I$-exact. Moreover, the unit $\eta_{B_0Y}:B_0Y\to UFB_0Y$ is a split monomorphism with cokernel isomorphic to $UH_1FY$.
\end{itemize}
\end{prop}

\begin{proof}
Factorize $d_1$ into $Y_1\twoheadrightarrow B_0Y\hookrightarrow Y_0$ and apply the functor $F$. This gives a factorization $FY_1\stackrel{g}{\twoheadrightarrow}FB_0Y\stackrel{h}{\to}FY_0$ of $Fd_1$ with $\im(h)=\im(hg)=\im(Fd_1)=B_0FY$. Since $g$ is surjective we have $\ker(h)=g(\ker(hg))=g(\ker(Fd_1))=g(Z_1FY)$. However, the exact sequence $Y_2\to Y_1\to B_0Y\to 0$ gives the exact sequence $FY_2\stackrel{Fd_2}{\to}FY_1\stackrel{g}{\to}FB_0Y\to 0$. This shows that $\ker(g)=\im(Fd_2)=B_1FY$. Therefore, $g$ induces an isomorphism $H_1FY=Z_1FY/B_1FY=Z_1FY/\ker(g)\cong g(Z_1FY)=\ker(h)$ and we get the required exact sequence
\[
0\to H_1FY\to FB_0Y\to B_0FY\to 0.
\]
If $Y$ is termwise projective then the sequence $FY_1\to FY_0\to FY_{-1}$ is $I$-exact. In fact, the unit $\eta_Y:Y\to UFY$ of the adjunction is an isomorphism of complexes (cf.\ Lemma \ref{unit_proj}). Therefore, Lemma \ref{I_exact} shows that the canonical map $B_0Y\cong B_0UFY\to UB_0FY$ is bijective. Consider the composed map $FY_1\twoheadrightarrow FB_0Y\twoheadrightarrow B_0FY$. If we apply the functor $U$ then the resulting map $Y_1\cong UFY_1\to UFB_0Y\to UB_0FY\cong B_0Y=\im(d_1)$ is given by $d_1$ and hence is surjective. This implies that the map $FB_0Y\to B_0FY$ is an $I$-epimorphism. Now apply $U$ to the exact sequence in (i) to get a map $UFB_0Y\to UB_0FY\cong B_0Y$ with kernel $UH_1FY$. One checks directly that it is a left inverse of the unit $\eta_{B_0Y}$.
\end{proof}

\begin{cor}\label{ess_surj}
If $M\in\mathrm{GProj}(H)$ and if $N\in\Mod(H)$ then the map $\Hom_H(N,M)\to\Hom_G(FN,FM)$ induced by $F$ is injective. If $H$ is Gorenstein and if $\Mod(H)$ is endowed with the Gorenstein projective model structure then the composed functor $\Rep^\infty_k(G)\stackrel{U}{\longrightarrow}\Mod(H)\longrightarrow\Ho(\Mod(H))$ is essentially surjective.
\end{cor}

\begin{proof}
As for the first statement it suffices to show the injectivity of the map $\Hom_H(N,M)\to\Hom_H(N,UFM)$ obtained by composing with the adjunction isomorphism $\Hom_G(FN,FM)\cong\Hom_H(N,UFM)$. Explicitly, this composition is given by sending $g:N\to M$ to $UFg\circ\eta_N=\eta_M\circ g$. If we realize $M=B_0Y$ for some acyclic complex $Y$ of projective $H$-modules then $\eta_M$ is a monomorphism by Proposition \ref{GP_split_mono} (ii). This proves the first statement.\\

In the homotopy category $\Ho(\Mod(H))$ any $H$-module becomes isomorphic to its cofibrant replacement. For the second statement it therefore suffices to show that any Gorenstein projective $H$-module is contained in the essential image of the functor $U$. This follows from Proposition \ref{GP_split_mono} (ii) because $M=B_0Y\cong UB_0FY$.
\end{proof}

The previous results can be interpreted in terms of Quillen adjunctions for suitable model structures on $\Rep^\infty_k(G)$ and $\Ch(G)$. For the rest of this section we endow $\Ch(H)$ with the singular projective model structure from Theorem \ref{Becker} (i). If $H$ is Gorenstein we will always endow $\Mod(H)$ with the Gorenstein projective model structure from Theorem \ref{GP_R_Mod}. As before let $F:\Mod(H)\rightleftarrows \Rep^\infty_k(G):U$ be the adjunction given by $FM=\XX\otimes_HM$ and $UV=V^I\cong\Hom_G(\XX,V)$. By the same symbols we denote its termwise extension to an adjunction $F:\Ch(H)\rightleftarrows \Ch(G):U$.

\begin{prop}\label{SP_right_transfer}
The right transfer of the singular projective model structure exists along the adjunction $F:\Ch(H)\rightleftarrows \Ch(G):U$ and makes it a Quillen equivalence. Endowing $\Ch(G)$ with the right transfer there is an equivalence of categories
\[ 
K_\ac(\Proj(H))\cong \Ho(\Ch(G)) .
\]
\end{prop}

\begin{proof}
Together with $\Rep^\infty_k(G)$ also $\Ch(G)$ is a Grothendieck category (cf.\ Lemma \ref{complexes_Grothendieck} (i)). In particular, all objects of $\Ch(G)$ are small. Since $\Ch(H)$ is cofibrantly generated the first result will follow from Theorem \ref{right_transfer} and Proposition \ref{path_objects} if we can show that path objects exist in $\Ch(G)$. To construct them let $V\in \Ch(G)$. We need a factorisation
$$
\Delta_V: V\stackrel{\alpha}{\longrightarrow} P \stackrel{\beta}{\longrightarrow}V\times V
$$
of the diagonal morphism such that $U\alpha$ is a weak equivalence and $U\beta$ is a fibration, i.e an epimorphism.\\

Let $C=\coker (\Delta_V)$ and let $\pi: V\times V\to C$ denote the canonical projection. Since the abelian category $\Ch(H)$ has enough projectives (cf.\ Lemma \ref{complexes_Grothendieck} (ii)) there is a surjection $f:Q\to UC$ where $Q$ is a projective complex. In terms of the singular projective model structure this is a fibration in which $Q$ is trivially cofibrant (cf.\ the proof of Theorem \ref{Becker}). Since $Q$ is termwise projective the unit of the adjunction $\eta_Q:Q \to UFQ$ is an isomorphism (cf.\ Lemma \ref{unit_proj}). Let $g:FQ\to C$ be the adjoint of $f$. Then $f=Ug\circ\eta_Q$ whence $Ug$ is surjective.\\

We now define $P=\{(v, q)\in (V\times V)\oplus FQ \;|\;\pi(v)=g(q)\}$ so that the projections $p:P\to V\times V$ and $h:P\to FQ$ make the square
\[
\begin{tikzcd}
P\arrow{r}{h}\arrow[swap]{d}{p} & FQ\arrow{d}{g}\\
V\times V\arrow[swap]{r}{\pi} & C
\end{tikzcd}
\]
cartesian. Let $i: V\to P$ be given by $i(v)=(\Delta (v), 0)$. This gives a factorization $\Delta_V=pi$ and we claim that it has the desired properties. As a right adjoint the functor $U$ preserves fibre products. Therefore, the square
\[
\begin{tikzcd}
UP\arrow{r}{Uh}\arrow[swap]{d}{Up} & UFQ\arrow{d}{Ug}\\
U(V\times V)\arrow[swap]{r}{U\pi} & UC
\end{tikzcd}
\]
is cartesian. Since the class of epimorphisms in an abelian category is stable under pullbacks and since $Ug$ is an epimorphism it follows that so is $Up$.\\

Note that $\Delta_V$ is a split monomorphism with splitting the first projection. Therefore, the image of the sequence $0\longrightarrow V\stackrel{\Delta_V}{\longrightarrow}V\times V\stackrel{\pi}{\longrightarrow} C\longrightarrow 0$ under $U$ remains exact. In particular, $U\pi$ is surjective and so is its pullback $Uh$. Since fiber products in abelian categories preserve kernels the sequence
$$
0\longrightarrow UV\stackrel{Ui}{\longrightarrow} UP\stackrel{Uh}{\longrightarrow} UFQ\longrightarrow 0
$$
is exact. Thus, $\coker (Ui)\cong UFQ\cong Q$ is trivially cofibrant. This shows that $Ui$ is a trivial cofibration and hence a weak equivalence, as required.\\

We have now shown that the right transfer exists. To get that the adjunction is a Quillen equivalence we simply apply Lemma \ref{unit_equiv} and Lemma \ref{unit_proj}. Recall that the cofibrant objects in $\Ch(H)$ are the acyclic complexes which are termwise projective (cf.\ Theorem \ref{Becker} (i)). The final statement follows from Theorem \ref{Becker} (i) and \cite{HovBook}, Proposition 1.3.13.
\end{proof}

The model structure on $\Ch(G)$ constructed in Proposition \ref{SP_right_transfer} will be referred to as the {\it $I$-singular projective model structure}.

\begin{cor}\label{ISP_stable}
The $I$-singular projective model structure on $\Ch(G)$ is stable. More precisely, the adjunction $[-1]:\Ch(G)\rightleftarrows\Ch(G):[1]$ is a Quillen equivalence whose derived adjunction is given by the loop and suspension functors. 
\end{cor}

\begin{proof}
The statement about the Quillen equivalence follows formally from Proposition \ref{SP_right_transfer} and Lemma \ref{explicit_loop_susp} because $U$ and $F$ commute with the shift functors. Note that $RU$ commutes with the loop functor by \cite{BR}, Corollary 3.1.4. Therefore, Lemma \ref{explicit_loop_susp} and the right derived version of \cite{HovBook}, Theorem 1.3.7, imply
\[
RU\Omega\cong\Omega RU\cong R[1]RU\cong R([1]U)\cong R(U[1])\cong RUR[1].
\]
Composing with $LF$ gives the isomorphism $\Omega\cong R[1]$ on $\Ch(G)$. The uniqueness of adjoints then also gives $L[-1]\cong\Sigma$. 
\end{proof} 

As in Lemma \ref{I_cofibrations} we are in a situation where the right adjoint $U$ preserves cofibrant objects. 

\begin{lem}\label{ISP_cofibrant}
Endow $\Ch(G)$ with the $I$-singular projective model structure. For an object $V_\bullet\in\Ch(G)$ the following are equivalent.
\begin{itemize}
\item[(i)]$V_\bullet$ is cofibrant
\item[(ii)]$V_\bullet$ is an $I$-exact sequence of $I$-projectives, i.e.\ $UV_\bullet=V_\bullet^I$ is an acyclic complex of projective $H$-modules.
\end{itemize}
If these conditions are satisfied then the counit $\varepsilon_{V_\bullet}:FUV_\bullet\to V_\bullet$ is an isomorphism in $\Ch(G)$.
\end{lem}

\begin{proof}
Any cofibrant object $Z\in\Ch(H)$ is termwise projective. By Lemma \ref{unit_proj} the unit $\eta_Z:Z\to UFZ$ is an isomorphism. With these observations and using the characterization of $I$-projectives in Lemma \ref{summand_I_proj} (iii) the proof of Lemma \ref{I_cofibrations} carries over almost verbatim.
\end{proof}

As a formal consequence we get the following analog of Corollary \ref{fib_cofib_equivalence}.

\begin{cor}\label{SP_fib_cofib}
Endow $\Ch(H)$ and $\Ch(G)$ with the singular projective model structure and the $I$-singular projective model structure, respectively. Then $U$ and $F$ restrict to inverse equivalences of categories $\Ch(G)_c\rightleftarrows\Ch(H)_c$.\hfill$\Box$
\end{cor}

Moreover, we obtain the following description of the cofibrations in $\Ch(G)$.

\begin{cor}\label{ISP_cofibrations} A morphism $i:V_\bullet\to W_\bullet$ in $\Ch(G)$ is a cofibration for the $I$-singular projective model structure if and only if it is a termwise split monomorphism whose cokernel is an $I$-exact sequence of $I$-projectives.
\end{cor}

\begin{proof}
Assume that $i$ is a cofibration. For any $n\in\mathbb{Z}$ consider the complex $D^{n+1}V_n=[0\to V_n\overset{\id}{\to}V_n\to 0]$ concentrated in degrees $n+1$ and $n$. Then $D^{n+1}V_n\to 0$ is an $I$-epimorphism and hence is a fibration. In fact, this is a trivial fibration because applying $U$ gives the bounded acyclic complex $[0\to UV_n\overset{\id}{\to}UV_n\to 0]$ which is trivial in the singular projective model structure by \cite{Beck14}, Proposition 3.1.1. Consider the map $V_\bullet\to D^{n+1}V_n$ given by the differential of $V_\bullet$ in degree $n+1$ and the identity of $V_n$ in degree $n$. By the left lifting property of $i$ applied to the diagram
\[
\xymatrix{
V_\bullet\ar[r]\ar[d]_i & D^{n+1}V_n\ar[d] \\
W_\bullet\ar[r] & 0
}
\]
there is a map $j:W_\bullet\to D^{n+1}V_n$ such that $j_n$ is a splitting of $i_n$. Since the cokernel of a cofibration is cofibrant (cf.\ \cite{HovBook}, Corollary 1.1.11) it follows together with Lemma \ref{ISP_cofibrant} that $i$ is a termwise split monomorphism whose cokernel is an $I$-exact sequence of $I$-projectives.\\

For the converse the arguments given in \cite{HovBook}, Proposition 2.3.9, carry over mutatis mutandis.
\end{proof}

We now pass back from complexes to the categories $\Mod(H)$ and $\Rep^\infty_k(G)$. For the rest of this section we will therefore assume that $H$ is a Gorenstein ring and that $\Mod(H)$ carries the Gorenstein projective model structure from Theorem \ref{GP_R_Mod}.

\begin{prop}\label{GP_right_transfer}
The right transfer of the Gorenstein projective model structure exists along the adjunction $F:\Mod(H)\rightleftarrows\Rep^\infty_k(G):U$.
\end{prop}

\begin{proof}
Since $\Mod(H)$ is cofibrantly generated this will follow from Theorem \ref{right_transfer} and Proposition \ref{path_objects} if we can show that path objects exist in $\Rep^\infty_k(G)$. To construct them we follow the argument in the proof of \cite{Hov2}, Proposition 9.1. Given $W\in\Rep^\infty_k(G)$ we choose an $I$-epimorphism $q:Y\to W$ where $Y$ is $I$-projective. Then the maps
\[
W\stackrel{i}{\longrightarrow}W\times Y\stackrel{p}{\longrightarrow}W\times W
\]
given by $i(w)=(w,0)$ and $p(w,y)=(w,w+q(y))$ give a factorization $\Delta_W=p i$ of the diagonal $\Delta_W:W\to W\times W$. Moreover, $Ui$ is an injection such that $\coker(Ui)\cong UY$ is a projective $H$-module. Thus, $Ui$ is a trivial cofibration in $\Mod(H)$. This implies that $i$ is a weak equivalence. Moreover, the map $Up$ is surjective because the map $Uq$ is. Thus, $Up$ is a fibration in $\Mod(H)$. This implies that $p$ is a fibration.
\end{proof}

Using Lemma \ref{enough_I_projectives} one can follow the above arguments to give a second proof of the existence of path objects in the situation of Proposition \ref{SP_right_transfer}. \\

The model structure on $\Rep^\infty_k(G)$ constructed in Proposition \ref{GP_right_transfer} will be called the {\it $I$-Gorenstein projective} model structure. This is the model structure on $\Rep^\infty_k(G)$ we will consider for the rest of this section. To describe it more concretely call an object $V\in\Rep^\infty_k(G)$ {\it $I$-trivial} if the object $UV\in\Mod(H)$ is trivial, i.e.\ has finite projective dimension. With this terminology the fibrations in $\Rep^\infty_k(G)$ are the $I$-epimorphisms and the trivial fibrations are the $I$-epimorphisms with $I$-trivial kernel.\\

We now give a different characterization of the $I$-trivial objects. Recall from the discussion following Corollary \ref{proj_equiv} that every $V\in \Rep^\infty_k(G)$ admits an $I$-resolution $X_\bullet\to V\to 0$. By Lemma \ref{pullback} and Lemma \ref{I_exact} this is a complex $X_\bullet$ of $I$-projective $G$-representations such that $UX_\bullet\to UV\to 0$ is exact.

\begin{defin}\label{I_projdim} We say that a representation $V\in \Rep^\infty_k(G)$ has {\it finite $I$-projective dimension} if it admits an $I$-resolution of finite length, i.e.\ if there exists an $I$-exact sequence
$0\to X_n\to X_{n-1}\to\cdots\to X_0\to V\to 0$
where $X_i$ is $I$-projective for any $i$.
\end{defin}

\begin{lem}\label{cond_trivial} A representation $V\in \Rep^\infty_k(G)$ is $I$-trivial if and only if it has finite $I$-projective dimension.
\end{lem}

\begin{proof}
Assume that $V$ has finite $I$-projective dimension. Applying $U$ to a finite $I$-resolution of $V$ we obtain a finite projective resolution of $UV$. Thus, $V$ is $I$-trivial. Conversely, suppose that $V$ is $I$-trivial and choose any $I$-resolution $X_\bullet=(X_\bullet,d_\bullet)$ of $V$. Let $n$ denote the projective dimension of $UV$. Since $UX_\bullet$ is a projective resolution of $UV$ the $H$-module $\ker(Ud_{n-1})$ is projective by \cite{Sta21}, 00O5. By Lemma \ref{I_exact} we have $\ker(Ud_{n-1})=U\im(d_n)$. Moreover, $\im (d_n)$ is generated by its $I$-invariants because it is a quotient of the $I$-projective object $X_n$ (cf.\ Lemma \ref{summand_I_proj}). Applying Lemma \ref{summand_I_proj} again we see that $\im(d_n)$ is $I$-projective and that
\[
0\to \im (d_n) \to X_{n-1}\to\ldots\to X_0\to V\to 0 
\]
is a finite $I$-resolution of $V$.
\end{proof}

In a next step we characterize the trivial cofibrations in $\Rep^\infty_k(G)$. 

\begin{lem}\label{triv_cof_split} A morphism $i:V\to W$ in $\Rep^\infty_k(G)$ is a trivial cofibration if and only if it is a split injection with an $I$-projective cokernel.
\end{lem}

\begin{proof}
If $i$ is a split injection with an $I$-projective cokernel we may assume $W=V\oplus C$ where $C$ is $I$-projective. If we have a commutative diagram
\[
\begin{tikzcd}
V\arrow{r}{f}\arrow{d}{i} & V'\arrow{d}{p}\\
W= V\oplus C\arrow{r} & W'
\end{tikzcd}
\]
where $p$ is an $I$-epimorphism then the map $C\to W'$ lifts to a map $g:C\to V'$ because $C$ is $I$-projective. The map $(f,g):V\oplus C=W\to V'$ is then a lift of $i$.\\

Conversely, suppose that $i:V\to W$ is a trivial cofibration. Its cokernel $C$ is trivially cofibrant by \cite{HovBook}, Corollary 1.1.11. This means that the functor $\Hom_G(C,\cdot)$ sends $I$-epimorphisms to surjections which is equivalent to $C$ being $I$-projective. It remains to show that $i$ is a split injection. Factorize $i$ as $V\stackrel{\alpha}{\to}V\oplus W\stackrel{p_2}{\to} W$ where $\alpha(v)=(v, i(v))$ and $p_2$ is the projection onto the second summand. Since $p_2$ is an $I$-epimorphism and since $i$ is a trivial cofibration we obtain a lift $h$ in the commutative diagram
\[
\xymatrix@R=0.75cm{
V\ar[r]^{\alpha}\ar[d]^{i} & V\oplus W \ar[d]^{p_2}\\
W\ar@{=}[r]\ar@{.>}[ur]^h & W.
}
\]
The first projection $p_1:V\oplus W\to V$ then gives the splitting $p_1 h$ of $i$.
\end{proof}

The same argument shows that if $V\in\Rep^\infty_k(G)$ is $I$-trivial then any cofibration $i:V\to W$ is a split injection with a cofibrant cokernel. In general, however, cofibrations will not be monomorphisms. Indeed, the adjunction in Proposition \ref{GP_right_transfer} is Quillen by Lemma \ref{unit_equiv}. Therefore, the functor $F$ preserves cofibrations, i.e.\ the image under $F$ of any monomorphism with a Gorenstein projective cokernel is a cofibration in $\Rep^\infty_k(G)$. By the lack of exactness of $F$ this will not be a monomorphism in general. However, we have the following result.

\begin{lem}\label{cofib_I_inv} If $i:V\to W$ is a cofibration then so is the inclusion $\im(i)\hookrightarrow W$ and we have $W=\im(i)+k[G]\cdot UW$. In particular, any cofibrant object is generated by its $I$-invariants.
\end{lem}

\begin{proof}
For the first part note that
\[
\xymatrix@R=0.75cm{
V\ar[r]^{i}\ar[d]^{i} & \im(i) \ar@{^{(}->}[d]\\
W\ar@{=}[r] & W
}
\]
is a pushout square and that the class of cofibrations is closed under pushouts (cf.\ \cite{HovBook}, Corollary 1.1.11). For the second part let $W'=\im(i)+k[G]\cdot UW$. The inclusion $W'\hookrightarrow W$ is an isomorphism on $I$-invariants, hence is a trivial fibration. Thus, we may apply the lifting property of $i$ to the diagram
\[
\xymatrix@R=0.75cm{
V\ar[r]^{i}\ar[d]^{i} & W'\ar[d]\\
W\ar@{=}[r]\ar@{.>}[ur] & W
}
\]
and get $W=W'$. The last statement follows from taking $V=0$.
\end{proof}

In order to characterize the cofibrant objects we first need to define a certain class of short exact sequences in $\Rep^\infty_k(G)$.

\begin{defin}\label{good_seq} Let $\mathcal{P}$ be the class of sequences $0\to V\to W\to X\to 0$ in $\Rep^\infty_k(G)$ which are both exact and $I$-exact. 
\end{defin}

A morphism $p:W\to X$ in $\Rep^\infty_k(G)$ is called a $\mathcal{P}$-epimorphism if the sequence $0\to \ker(p)\to W\stackrel{p}{\to}X\to 0$ belongs to $\mathcal{P}$. Note that the $\mathcal{P}$-epi\-mor\-phisms are precisely the epimorphisms which are also $I$-epimorphisms. The notion of a $\mathcal{P}$-mono\-mor\-phism is defined dually. In fact, the class $\mathcal{P}$ is proper in the sense of \cite{Mac}, Chapter XII.4.

\begin{lem}\label{good_proper} The class $\mathcal{P}$ is proper.
\end{lem}

\begin{proof}
Clearly $\mathcal{P}$ is closed under isomorphisms of short exact sequences and contains all split exact sequences. Now let $f:V\to V'$ and $g:V'\to V''$ be a composable pair of morphisms. Assume first that $f$ and $g$ are epimorphisms. We must show that if $gf$ is an $I$-epimorphism then so is $g$ and that if $f$ and $g$ are $I$-epimorphisms then so is $gf$. Both cases follow immediately by applying $U$ to $gf$. Now assume that $f$ and $g$ are monomorphisms. We may use them to identify $V$ and $V'$ with subrepresentations of $V''$. Consider the commutative diagram
\[
\xymatrix@R=0.5cm{
 0\ar[r] & (V')^I\ar[r] \ar[d] & (V'')^I\ar[r] \ar[d] & (V'')^I/(V')^I\ar[r]\ar[d]&0\\
 0\ar[r] & (V'/V)^I \ar[r] & (V''/V)^I\ar[r] & (V''/V')^I&
}
\]
with exact rows. If $f$ and $g$ are $\mathcal{P}$-monomorphisms then the outer vertical maps are surjective. The snake lemma implies that the vertical map in the middle is surjective, too. This means that $gf$ is a $\mathcal{P}$-monomorphism. Finally, if $gf$ is a $\mathcal{P}$-monomorphism then the vertical map in the middle is surjective. Since the right vertical map is always a monomorphism the snake lemma implies that the left vertical map is surjective, too. This means $f$ is a $\mathcal{P}$-monomorphism.
\end{proof}

Using the proper class $\mathcal{P}$ we can characterize the cofibrant objects of $\Rep^\infty_k(G)$ by applying the arguments from \cite{Hov2}, Proposition 4.1.

\begin{lem}\label{cofibrant_ext} For an object $V\in \Rep^\infty_k(G)$ the following are equivalent.
\begin{itemize}
\item[(i)] $V$ is cofibrant in the $I$-Gorenstein projective model structure;
\item[(ii)] $V$ is generated by its $I$-invariants and $\Ext^1_{\mathcal{P}}(V,W)=0$ for all objects $W\in \Rep^\infty_k(G)$ of finite $I$-projective dimension.
\end{itemize}
\end{lem}

\begin{proof}
Assuming (i) $V$ is generated by its $I$-invariants (cf.\ Lemma \ref{cofib_I_inv}). Let $W\in \Rep^\infty_k(G)$ have finite $I$-projective dimension. Any element of $\Ext^1_{\mathcal{P}}(V,W)$ is the class of a short exact sequence $0\to W\to E\stackrel{p}{\to}V\to 0$ in $\mathcal{P}$. Then $Up$ is a surjection whose kernel $UW$ has finite projective dimension (cf.\ Lemma \ref{cond_trivial}). Thus, $p$ is a trivial fibration. Since $V$ is cofibrant it follows that $p$ admits a section. This implies $\Ext^1_{\mathcal{P}}(V,W)=0$.\\

Assuming (ii) let $p:V'\to V''$ be a trivial fibration. To show that $V$ is cofibrant we must prove that any morphism $f:V\to V''$ lifts to a morphism $V\to V'$ along $p$. Note that $k[G]\cdot UV''\subseteq \im(p)$ because $p$ is an $I$-epimorphism. Since $V$ is assumed to be generated by its $I$-invariant we get $\im(f)\subseteq \im(p)$. Thus, we may assume that $V''=\im(p)$ and that $p$ is a $\mathcal{P}$-epimorphism. If we set $W=\ker(p)$ then the short exact sequence $0\to W\to V'\stackrel{p}{\to} V''\to 0$ belongs to $\mathcal{P}$. Since $p$ is a trivial fibration $W$ has finite $I$-projective dimension by Lemma \ref{cond_trivial}. By the long exact sequence on $\Ext_{\mathcal{P}}$ we get that $\Hom_G(V,V')\to\Hom_G(V,V'')\to \Ext^1_{\mathcal{P}}(V,W)$ is exact. Our vanishing assumption therefore implies that $f$ lifts to a map $V\to V'$ as required. 
\end{proof}

\begin{rem} As in Remark \ref{relative_ext_ambiguous} the group $\Ext^1_{\mathcal{P}}(V,W)$ is not to be confused with the relative extension group $\Ext^1_{G,I}(V,W)$ defined in (\ref{relative_Ext}). Since we have a functorial isomorphism $\Ext^1_{G,I}(V,W)\cong \Ext^1_H(V^I, W^I)$ (cf.\ Corollary \ref{Ext_comparison}) there is however a canonical group homomorphism $\Ext^1_{\mathcal{P}}(V,W)\to \Ext^1_{G,I}(V,W)$ sending an equivalence class $[0\to W\to E\to V\to 0]$ to the class $[0\to UW\to UE\to UV\to 0]$.
\end{rem}

We are now able to describe the following class of cofibrations by arguing as in the proof of \cite{Hov2}, Proposition 4.2.

\begin{prop}\label{good_monos_cofib} Any $\mathcal{P}$-monomorphism $i:V\to W$ with a cofibrant cokernel is a cofibration.
\end{prop}

\begin{proof}
Let $p:V'\to V''$ be a trivial fibration in a commutative square
\[
\xymatrix{
V\ar[r]^f\ar[d]_i & V'\ar[d]^p\\
W\ar[r]_g & V''.}
\]
Note that $W=\im(i)+k[G]\cdot UW$ because $W\to \coker(i)$ is an $I$-epimorphism and $\coker(i)$ is generated by its $I$-invariants (cf.\ Lemma \ref{cofib_I_inv}). Since $p$ is an $I$-epimorphism we get $\im(g)\subseteq\im(p)$. As in the previous proof we may therefore assume that $V''=\im(p)$ and that $p$ is a $\mathcal{P}$-epimorphism. Write $K=\ker(p)$ and $C=\coker(i)$. Since $i$ is a $\mathcal{P}$-monomorphism we can consider long exact sequences to obtain the following commutative diagram whose rows and columns are exact.
\[
\begin{tikzcd}
\Hom_G(W,K)\arrow{r}\arrow{d} & \Hom_G(W,V')\arrow{d} \arrow{r} & \Hom_G(W,V'')\arrow{d}{i^*}\\
\Hom_G(V,K)\arrow{r}\arrow{d} & \Hom_G(V,V')\arrow{d}{\delta}\arrow{r}{p_*} & \Hom_G(V,V'')\arrow{d}{\delta}\\
\Ext^1_{\mathcal{P}}(C,K)\arrow{r} & \Ext^1_{\mathcal{P}}(C,V')\arrow{r}{p_*} & \Ext^1_{\mathcal{P}}(C,V'')
\end{tikzcd}
\]
Since $C$ is cofibrant and $K$ has finite $I$-projective dimension (cf.\ Lemma \ref{cond_trivial}) we have $\Ext^1_{\mathcal{P}}(C,K)=0$ by Lemma \ref{cofibrant_ext}. Now the maps $f$ and $g$ satisfy $p_*f=i^*g$ and therefore $p_*\delta f=\delta p_*f=\delta i^* g=0$. This implies $\delta f=0$ whence there exists $h: W\to V'$ with $f=h i$.\\

By construction $g-p h:W\to V''$ is zero on $\im(i)$ and therefore factors through a map $\alpha: C\to V''$. Since $C$ is cofibrant and $p$ is a trivial fibration $\alpha$ lifts to a map $\beta: C\to V'$ such that $p \beta=\alpha$. Precomposing $\beta$ with the quotient map $W\to C$ we obtain a map $j:W\to V'$ satisfying $p j=g-p h$. Therefore, the sum $h+j:W\to V'$ gives our desired lift of $i$.
\end{proof}

\begin{rem} By definition of the Gorenstein projective model structure (cf.\ Theorem \ref{GP_R_Mod}) the cofibrations in $\Mod(H)$ are precisely the inclusions $\ker(q)\to M$ where $M$ is arbitrary and $q:M\to N$ is a fibration onto a cofibrant object $N$. Using the above result we see that the analogous maps in $\Rep^\infty_k(G)$ are also cofibrations. Indeed, if $X\in\Rep^\infty_k(G)$ is cofibrant and if $p:W\to X$ is an $I$-epimorphism then $p$ is a $\mathcal{P}$-epimorphism because $X$ is generated by its $I$-invariants (cf.\ Lemma \ref{cofib_I_inv}). Therefore, the inclusion $\ker(p)\to W$ is a $\mathcal{P}$-monomorphism with cokernel $X$ and so is a cofibration.
\end{rem}

Recall that the adjunction in Proposition \ref{GP_right_transfer} is Quillen. We now investigate how far it is from being a Quillen equivalence. In a first step we make explicit the homotopy relation in $\Rep^\infty_k(G)$ following \cite{Hov2}, Proposition 9.1.

\begin{lem}\label{IGP_homotopy}Let $V,W\in\Rep^\infty_k(G)$ with $V$ cofibrant. If $f,g\in\Hom_G(V,W)$ then $f$ and $g$ are homotopic if and only if $g-f$ factors through an $I$-projective. 
\end{lem}

\begin{proof}
Since $W$ is fibrant it follows from \cite{HovBook}, Proposition 1.2.5 (v), that $f$ and $g$ are homotopic if and only if they are right homotopic. So assume that $f$ and $g$ are right homotopic. By \cite{HovBook}, Corollary 1.2.6, there is a homotopy from $f$ to $g$ through any path object of $W$. In particular, we may take the path object $W\stackrel{i}{\to}W\times Y\stackrel{p}{\to}W\times W$ of $W$ constructed in the proof of Proposition \ref{GP_right_transfer}. Recall that this involves an $I$-epimorphism $q:Y\to W$ where $Y$ is $I$-projective. There is then a map $H=(r,s):V\to W\times Y$ with $p H=(f,g)$. By construction of $p$ this means $f=r$ and $g=r+q s$. Thus, $g-f=q s$ factors through the $I$-projective $Y$.\\

Conversely, assume that $g-f$ admits a factorization $V\to Z\stackrel{j}{\to} W$ where $Z$ is $I$-projective. Then $j$ factors as $Z\to Y\stackrel{q}{\to}W$ because $q$ is an $I$-epimorphism. Consequently, $g-f=q h$ where $h$ denotes the composition $V\to Z\to Y$. Since $g=f+q h$ we see that $f$ and $g$ are right homotopic by means of $H=(f,h):V\to W\times Y$.
\end{proof}

As a consequence we obtain that the statements in Corollary \ref{ess_surj} pass to the homotopy level. We continue to endow $\Mod(H)$ with the Gorenstein projective model structure and $\Rep^\infty_k(G)$ with the $I$-Gorenstein projective model structure.

\begin{thm}\label{LF_faithful}
In the derived adjunction
\[
LF:\Ho(\Mod(H))\rightleftarrows\Ho(\Rep_k^\infty(G)):RU
\]
the functor $LF$ is faithful and the functor $RU$ is essentially surjective.
\end{thm}

\begin{proof}
That $RU$ is essentially surjective follows directly from Corollary \ref{ess_surj}. In order to show that $LF$ is faithful let $Q_c$ denote the cofibrant replacement functor of $\Mod(H)$. Recall from (\ref{left_derived}) that $LF=\Ho(F)\Ho(Q_c)$. By the proof of \cite{HovBook}, Proposition 1.2.3, the functor $\Ho(Q_c)$ is always an equivalence. Therefore, it suffices to show that $\Ho(F)$ is faithful. Note that $F$ preserves cofibrant objects because it is left Quillen. Since all objects of $\Mod(H)$ and $\Rep^\infty_k(G)$ are fibrant it suffices to see that if $M,N\in\Mod(H)$ are Gorenstein projective then the map $\Hom_H(M,N)\to\Hom_G(FM,FN)$ induced by $F$ reflects the respective homotopy relations (cf.\ \cite{HovBook}, Theorem 1.2.10 (ii)).\\

Let $f,g\in\Mod_H(M,N)$ and assume that $Ff$ is homotopic to $Fg$. By Lemma \ref{IGP_homotopy} there is an $I$-projective object $Y\in\Rep^\infty_k(G)$ such that $Ff-Fg=F(f-g)$ admits a factorization $FM\to Y\to FN$. Applying $U$ we obtain the commutative diagram
\[
\xymatrix{
UFM\ar[r]&UY\ar[r]&UFN\\
M\ar[u]^{\eta_M}\ar[rr]^{f-g}&&N.\ar[u]_{\eta_N}
}
\]
Since $\eta_N$ admits a left inverse (cf.\ Proposition \ref{GP_split_mono}) it follows that $f-g$ factors through $UY$. However, $UY$ is a projective $H$-module (cf.\ Lemma \ref{summand_I_proj}). Therefore, $f$ and $g$ are homotopic by \cite{Hov2}, Proposition 9.1.
\end{proof}

In Theorem \ref{derived_right_inverse} the previous results will be strengthened significantly. In fact, $RU$ admits a right inverse and becomes an equivalence when restricted to a suitable subcategory of $\Ho(\Rep^\infty_k(G))$.\\

In order to compute the loop and suspension functors on $\Ho(\Rep^\infty_k(G))$ we need to compute both cylinder and path objects.

\begin{lem}\label{cylinder_obj} Let $V,W\in\Rep^\infty_k(G)$ and suppose that $W$ is $I$-projective.
\begin{itemize}
\item[(i)]Given a cofibration $q:V\to W$ the factorization $V\oplus V\stackrel{i}{\to} V\oplus W\stackrel{p}{\to}V$ given by $i(v,v')=(v+v', q(v'))$ and $p(v,w)=v$ exhibits $V\oplus W$ as a cylinder object for $V$. In particular, if $V$ is cofibrant then we have $\Sigma V\cong \coker(i)\cong\coker(q)$ in $\Ho(\Rep^\infty_k(G))$.
\item[(ii)]Given a fibration $q:W\to V$ the factorization $V\stackrel{i}{\to}V\times W\stackrel{p}{\to}V\times V$ given by $i(v)=(v,0)$ and $p(v,w)=(v,v+q(w))$ exhibits $V\times W$ as a path object for $V$. In particular, we have $\Omega V\cong\ker(p)\cong\ker(q)$ in $\Ho(\Rep^\infty_k(G))$.
\end{itemize}
\end{lem}

\begin{proof}
As for (i) we note that $p$ is a split epimorphism with an $I$-trivial kernel, hence is a trivial fibration. In order to see that the above factorization gives a cylinder object for $V$ we need to show that $i$ is a cofibration. Observe that the  diagram
\[
\begin{tikzcd}
V \arrow{r}{\delta} \arrow[swap]{d}{q} & V\oplus V\arrow{d}{i}\\
W\arrow[swap]{r}{j} &V\oplus W
\end{tikzcd}
\]
given by $\delta(v)=(-v,v)$ and $j(w)=(0,w)$ is a pushout square. Thus, together with $q$ also $i$ is a cofibration (cf.\ \cite{HovBook}, Corollary 1.1.11). If $V$ is cofibrant then $\Sigma V\cong\coker(i)$ in $\Ho(\Rep^\infty_k(G))$ by definition of the suspension. However, the map $\coker(q)\to \coker(i)$ induced by $j$ is an isomorphism in $\Rep^\infty_k(G)$ by \cite{Sta21}, 08N3.\\

The construction of the path object in (ii) is taken from the proof of Proposition \ref{GP_right_transfer}. Since $V$ is fibrant we have $\Omega V\cong\ker(p)$ in $\Ho(\Rep^\infty_k(G))$ by definition of the loop. However, the map $\ker(q)\to\ker(p)$ given by $w\mapsto (0,w)$ is an isomorphism in $\Rep^\infty_k(G)$.
\end{proof}

Because of condition (iii) in the following result we do not expect that the adjunction $F:\Mod(H)\rightleftarrows\Rep^\infty_k(G):U$ is a Quillen equivalence in general.

\begin{prop}\label{characterization_equiv}
The following statements are equivalent.
\begin{itemize}[wide]
\item[(i)]The adjunction $F:\Mod(H)\rightleftarrows\Rep^\infty_k(G):U$ is a Quillen equivalence.
\item[(ii)]The left derived functor $LF:\Ho(\Mod(H))\to\Ho(\Rep^\infty_k(G))$ is full.
\item[(iii)]For any acyclic complex $Y\in\Ch(H)$ of projective $H$-modules the $G$-representation $H_1FY$ is $I$-trivial, i.e.\ the $H$-module $UH_1FY$ has finite projective dimension.
\item[(iv)]The model category $\Rep^\infty_k(G)$ is stable.
\end{itemize}
\end{prop}

\begin{proof}
That (i) implies (ii) follows from \cite{HovBook}, Proposition 1.3.13. Assuming (ii) the functor $LF$ is fully faithful by Theorem \ref{LF_faithful}. This implies that the unit of the derived adjunction is an isomorphism. Let $X\in\Mod(H)$ be cofibrant and denote by $f:FX\to Z$ the fibrant replacement of $FX$ in $\Rep^\infty_k(G)$. By the proof of \cite{HovBook}, Proposition 1.3.13, the map $X\stackrel{\eta_X}{\to}UFX\stackrel{Uf}{\to}UZ$ is a weak equivalence. However, since $f$ is a weak equivalence so is $Uf$ by definition of the right transfer. By the 2-out-of-3 property $\eta_X$ is a weak equivalence for any cofibrant object $X$. By Lemma \ref{unit_equiv} this implies (i).\\

If a cofibrant object $X$ is realized as $X=B_0Y$ for some acyclic complex $Y$ of projective $H$-modules then $\eta_X$ is a monomorphism with cokernel $UH_1FY$ (cf.\ Proposition \ref{GP_split_mono}). By \cite{Hov2}, Lemma 5.8, the map $\eta_X$ is a weak equivalence if and only if the $H$-module $UH_1FY$ has finite projective dimension. Therefore, the above arguments show that (ii) is equivalent to (iii). Assuming (i) the functors $LF$ and $RU$ are equivalences of categories commuting with $\Sigma$ and $\Omega$, respectively (cf.\ \cite{BR}, Corollary 3.1.4). Since $\Mod(H)$ is stable (cf.\ \cite{Hov2}, Theorem 9.3) it follows formally that so is $\Rep^\infty_k(G)$.\\

Note that the unit $\id\to\Omega\Sigma$ of the adjunction in $\Ho(\Rep^\infty_k(G))$ gives rise to a natural transformation $LF\Omega\to \Omega\Sigma LF\Omega\cong \Omega LF\Sigma\Omega\cong \Omega LF$ using the stability of $\Mod(H)$. To describe this more explicitly let $Y=(Y_\bullet,d_\bullet)$ be an acyclic complex of projective $H$-modules and set $M=B_{-1}Y$. Then $\Omega M\cong B_0Y$ in $\Ho(\Mod(H))$ by the construction of the loop and suspension functors on abelian model categories. Since $B_0Y$ is cofibrant this gives $LF\Omega M\cong FB_0Y$. Similarly, $LFM\cong FM$ and therefore $\Omega LFM\cong\Omega FM$. Note that the canonical map $q:FB_0Y\to FY_0$ is a cofibration and $FB_0Y$ is cofibrant since $F$ is left Quillen. Therefore, Lemma \ref{cylinder_obj} (i) shows that $FB_0Y\oplus FY_0$ is a cylinder object for $FB_0Y$ and $\Sigma FB_0Y\cong \coker(q)\cong FM$. The corresponding adjoint morphism $\varphi:FB_0Y\to\Omega FM$ can then be identified with the aforementioned natural transformation $LF\Omega M\to \Omega LFM$. Assuming (iv) the map $\varphi$ is an isomorphism. We will show this implies (iii).\\

By the proof of Lemma \ref{cylinder_obj} (i) the isomorphism $\Sigma FB_0Y\cong FM$ is induced by the surjection $g:FB_0Y\oplus FY_0\to FM$ sending $(v,w)$ to $Fd_0(w)$. In order to compute the loop of $FM$ choose an $I$-epimorphism $q':Z\to FM$ where $Z$ is $I$-projective. By Lemma \ref{cylinder_obj} (ii) we have $\Omega FM=\ker(p)$ where $p:FM\times Z\to FM\times FM$ is given by $p(v,w)=(v,v+q'(w))$. Consider the commutative square
\[
\begin{tikzcd}
FB_0Y\arrow{r}{0} \arrow[swap]{d}{\iota} & FM\times Z\arrow{d}{p}\\
FB_0Y\oplus FY_0\arrow[swap]{r}{(0,g)} & FM\times FM
\end{tikzcd}
\]
where $\iota$ is the canonical map. Since $\iota$ is a trivial cofibration (cf.\ Lemma \ref{triv_cof_split}) and since $p$ is a fibration there is a map $H:FB_0Y\oplus FY_0\to FM\times Z$ such that $H \iota=0$ and $(0,g)=pH$. Consider the map $\iota':FB_0Y\to FB_0Y\oplus FY_0$ given by $\iota'(v)=(0,q(v))$. Since $FY$ is a complex the map $H\iota'$ factors through a map $f:FB_0Y\to\ker(p)$.  By the proof of \cite{BR}, Proposition 3.1.7, its homotopy class is the adjoint morphism $\varphi$. Note that together with $\iota'$ also $f$ factors through the canonical map $FB_0Y\to B_0FY$.\\

Since $\varphi$ is an isomorphism in $\Ho(\Rep^\infty_k(G))$ it follows from \cite{HovBook}, Theorem 1.2.10 (iv), that $f$ is a weak equivalence. By definition of the right transfer $Uf$ is a weak equivalence in $\Mod(H)$. As seen above it factors through the natural map $UFB_0Y\to UB_0FY$. However, the proof of Proposition \ref{GP_split_mono} shows that the latter is a split surjection with kernel $UH_1FY$. Setting $M_1=UB_0FY$ and $M_2=UH_1FY$ we identify it with the first projection $M_1\times M_2\to M_1$. Write $Uf=rj$ where $j$ is a trivial cofibration and $r$ is a trivial fibration. Then $j$ is a monomorphism with a projective cokernel and may be identified with a canonical inclusion of the form $M_1\times M_2\hookrightarrow P\times M_1\times M_2$ where $P$ is projective. Since $M_2\hookrightarrow M_1\times M_2$ is contained in the kernel of $Uf$ it follows that $M_2$ is a direct summand of $\ker(r)$. Since $\ker(r)$ has finite projective dimension so does $M_2=UH_1FY$.
\end{proof}

The condition in Proposition \ref{characterization_equiv} (iv) suggests to pass to some stable version of $\Rep^\infty_k(G)$. The most direct approach is to define the {\it stable homotopy category} $\underline{\Ho}(\Rep^\infty_k(G))$ of $\Rep^\infty_k(G)$ to have the same objects as $\Rep^\infty_k(G)$ and to define the set of morphisms between $X$ and $Y$ by
\[
\underline{[X,Y]}=\varinjlim_{n\geq 0}[\Omega^nX,\Omega^nY].
\]
The transition maps in this direct limit are induced by $\Omega$ and $[\Omega^nX,\Omega^nY]$ denotes the set of morphisms in the usual homotopy category. The loop functor induces a fully faithful functor $\underline{\Omega}$ on $\underline{\Ho}(\Rep^\infty_k(G))$. Moreover, since $RU$ commutes with $\Omega$ and since $\Mod(H)$ is stable there is an induced functor
\[
\underline{RU}:\underline{\Ho}(\Rep^\infty_k(G))\to\Ho(\Mod(H))
\]
given by $RU$ on objects and sending the class of $f\in[\Omega^nX,\Omega^nY]$ to $\Sigma^nRUf$. However, we currently do not know if this improves the properties of $RU$ further.


\section{Frobenius categories}\label{section_5}

%
%


Recall that a sequence $X\stackrel{i}{\to}Y\stackrel{p}{\to}Z$ in an additive category $\A$ is called {\it exact\- } if $i$ is a kernel of $p$ and if $p$ is a cokernel of $i$. An {\it exact category} is an additive category $\A$ endowed with a class of exact sequences satisfying the axioms in \cite{Bue10}, Definition 2.1. For the sake of clarity these exact sequences are sometimes called {\it admissible} and the morphism $i$ (resp.\ $p$) in an admissible exact sequence $X\stackrel{i}{\to}Y\stackrel{p}{\to}Z$ is called an {\it admissible monomorphism} (resp.\ {an \it admissible epimorphism}).\\

An object $Y$ of an exact category $\A$ is called {\it projective} (resp.\ {\it injective}) if the functor $\Hom_\A(Y,\cdot)$ (resp.\ $\Hom_\A(\cdot,Y)$) is {\it exact}, i.e.\ if it transforms admissible exact sequences into exact sequences of abelian groups. An exact category $\A$ is said to have {\it enough projectives} (resp.\ {\it injectives}) if for any object $X$ there is an admissible epimorphism $Y\to X$ (resp.\ an admissible monomorphism $X\to Y$) where $Y$ is projective (resp.\ injective). An exact category $\A$ with enough projectives and enough injectives is called a {\it Frobenius category} if the classes of projective and injective objects of $\A$ coincide.\\

In the following we will often assume that the exact category $\A$ is {\it weakly idempotent complete} in the sense of \cite{Bue10}, Definition 7.2. By \cite{Gil11}, Proposition 2.4, this is equivalent to the classes of admissible monomorphisms and admissible epimorphisms being closed under retracts. Hovey's correspondence in Theorem \ref{cotorsion_pairs} was generalized to weakly idempotent complete exact categories by Gillespie in \cite{Gil11}, Corollary 3.4. The corresponding model structures are called {\it exact} (cf.\ \cite{Gil11}, Definition 3.1). We note that the completeness of the cotorsion pairs in \cite{Gil11}, Definition 2.1, is not required to be functorial and that the bicompleteness assumption on $\A$ is relaxed significantly (cf.\ the beginning of \cite{Gil11}, \S4).\\

Recall the following standard result (cf. \cite{Li}, Theorem 1.1 and Remark 2.6).

\begin{prop}\label{Frobenius_model_structure}
Let $\A$ be a weakly idempotent complete Frobenius category. There is a unique exact model structure on $\A$ for which
\begin{itemize}
\item all objects are fibrant and cofibrant,
\item the trivial objects are the projective (equivalently, the injective) objects.
\end{itemize}
This model structure is stable and its associated homotopy category $\Ho(\A)$ is the stable category $\underline{\A}=\A/\Proj(\A)$ of $\A$. 
\end{prop}

\begin{proof}
Set $\T=\Proj(\A)=\Inj(\A)$. By definition of a Frobenius category $(\A,\T)$ and $(\T,\A)$ are complete cotorsion pairs on $\A$ in the sense of \cite{Gil11}, \S2.1. Since $\T$ is a thick subcategory of $\A$ in the sense of \cite{Gil11}, Definition 3.2, the existence of the model structure follows from \cite{Gil11}, Corollary 3.4. The description of the corresponding homotopy category is given in \cite{Gil11}, Proposition 4.3 (5). That the suspension functor on $\underline{\A}$ is an equivalence of categories is proved in \cite{Hap}, Proposition 2.2. 
\end{proof}

If $\A$ is a (weakly idempotent complete) exact category then so is $\Ch(\A)$ with respect to those sequences of complexes which are admissible in every degree (cf.\ \cite{Gil16}, Lemma 2.5). The {\it acyclicity} of complexes in $\Ch(\A)$ is defined as in \cite{Bue10}, Definition 10.1. Moreover, if $\A$ is a Frobenius category then so is $\Ch(\A)$ (cf.\ \cite{Gil16}, Corollary 2.7).\\

Our main case of interest concerns the category $\A=\mathrm{GProj}(S)$ where $S$ is a Gorenstein ring. Viewed as a full subcategory of $\Mod(S)$ with the induced exact structure this is a Frobenius category by \cite{Gil11}, Proposition 5.2 (4). It is weakly idempotent complete because the class of Gorenstein projective modules is closed under retracts. In this situation we also have the following variant of Theorem \ref{Becker}.

\begin{prop}\label{sing_proj_Frobenius}
If $S$ is a Gorenstein ring and if $\A=\mathrm{GProj}(S)$ then there is an exact model structure on $\Ch(\A)$ for which
\begin{itemize}
\item the cofibrant objects are the acyclic complexes of projective $S$-modules,
\item all objects are fibrant,
\item the trivial objects are the complexes $Y\in\Ch(\A)$ such that we have $\Ext^1_{\Ch(\A)}(X,Y)=0$ for any acyclic complex of projectives $X$.
\end{itemize}
The adjunction $Q_0:\Ch(\A)\leftrightarrows\A:\iota_0$ is a Quillen equivalence.
\end{prop}

\begin{proof}
Let $\mathcal{B}$ denote the class of $S$-modules of finite projective dimension. By \cite{Hov2}, Theorem 8.3, $(\A,\mathcal{B})$ is a functorially complete cotorsion pair on $\Mod(S)$ which is generated by a set. Moreover, $\A$ generates the category $\Mod(S)$ because it contains the projective generator $S$. It follows from \cite{Beck14}, Proposition 1.2.5, that the class $\A$ is {\it deconstructible} in the sense of \cite{St14}, Definition 3.9. By \cite{St14}, Lemma 7.9, the class of acyclic complexes in $\Ch(\A)$ is deconstructible in $\Ch(\A)$. We now apply \cite{St14}, Lemma 7.10, to the complete cotorsion pair $(\mathrm{Proj}(S),\A)$ of $\A$ (cf.\ Proposition \ref{Frobenius_model_structure}). Letting $\C$ denote the class of acyclic complexes of projective $S$-modules and
\[
\C^\perp=\{Y\in\Ch(\A)\;|\;\Ext^1_{\Ch(\A)}(X,Y)=0\mbox{ for all }X\in\ \C\}
\]
we obtain that $(\C,\C^\perp)$ is a complete cotorsion pair in $\Ch(\A)$. Note that $\Ext^1_{\Ch(\A)}=\Ext^1_{\Ch(S)}$ as a bifunctor on $\Ch(\A)$ because the exact subcategory $\Ch(\A)$ of $\Ch(S)$ is closed under extensions. Therefore, it follows from Theorem \ref{Becker} (i) that $\C\cap\C^\perp$ is the class of projective objects of $\Ch(S)$. Note that this is contained in $\Ch(\A)$. Applying \cite{Gil11}, Corollary 3.4, to the triple $(\C,\C^\perp,\Ch(\A))$ yields the required exact model structure. As a consequence of Theorem \ref{Becker} (i) and Theorem \ref{cotorsion_pairs} (i) the corresponding cotorsion pairs are functorially complete. The final assertion follows from Theorem \ref{Becker} (ii) because $\Ch(\A)$ and $\Ch(S)$ have the same class of cofibrant-fibrant objects.
\end{proof}

We now return to the situation where $S=H=\End_G(\XX)^\op$. For the rest of this section we assume that $H$ is Gorenstein and endow the categories $\Mod(H)$ and $\Ch(H)$ with the Gorenstein projective and the singular projective model structure, respectively (cf.\ Theorem \ref{GP_R_Mod} and Theorem \ref{Becker} (i)). Moreover, we endow the categories $\Rep^\infty_k(G)$ and $\Ch(G)$ with the model structures obtained via the right transfer along the adjoint pair $(F,U)$ (cf.\ Proposition \ref{SP_right_transfer} and Proposition \ref{GP_right_transfer}).

\begin{defin}\label{Cabanes_category}
A representation $V\in\Rep_k^\infty(G)$ is called {\it $I$-Gorenstein projective} if $V$ is isomorphic to $B_0X=\im(d_1)$ for some $I$-exact sequence $X=(X_\bullet,d_\bullet)$ of $I$-projective $G$-representations. We denote by $\C(G)$ the full subcategory of $\Rep^\infty_k(G)$ consisting of all $I$-Gorenstein projective $G$-representations.
\end{defin}

It follows from Lemma \ref{ISP_cofibrant} and Corollary \ref{SP_fib_cofib} that a representation $V$ is $I$-Gorenstein projective if and only if there is an acyclic complex $Y=(Y_\bullet,d_\bullet)$ of projective $H$-modules such that $V\cong B_0FY=\im(Fd_1)$.\\

Using the functorial factorizations in the model category $\Mod(H)$ one can construct functorial complete resolutions as in \cite{Gil21}, \S1. The outcome is a functor $Y^{(\cdot)}=(M\mapsto Y^M=(Y^M_\bullet,d^M_\bullet)):\Mod(H)\to\Ch(H)$ such that 
\begin{itemize}
\item the complex $Y^M\in\Ch(H)$ is acyclic with $B_0Y^M=Z_0Y^M=M$,
\item for any $i\geq 1$ the $H$-module $Y^M_i$ is projective,
\item for any $i\leq 0$ the $H$-module $Y^M_i$ has finite projective dimension.
\end{itemize}
Note that our numbering of the complex $Y^M$ differs from that in \cite{Gil21} by a shift. An $H$-module $M$ is Gorenstein projective if and only if $Y^M_i$ is a projective $H$-module for all $i\in\mathbb{Z}$ (cf.\ \cite{Gil21}, Lemma 4.3 (4)). We now define the functor $\FF:\Mod(H)\to\Rep^\infty_k(G)$ as the composition
\[
\Mod(H)\stackrel{Y^{(\cdot)}}{\longrightarrow}\Ch(H)\stackrel{F}{\longrightarrow}\Ch(G)\stackrel{B_0}{\longrightarrow}\Rep^\infty_k(G),
\]
sending $M\in\Mod(H)$ to $\FF M=B_0FY^M\in\Rep^\infty_k(G)$. 

\begin{thm}\label{Cabanes}
The functors
\[
U:\Rep^\infty_k(G)\to\Mod(H)\quad\mbox{and}\quad\FF:\Mod(H)\to\Rep^\infty_k(G)
\]
restrict to inverse equivalences of categories $\C(G)\cong\mathrm{GProj}(H)$.
\end{thm}

\begin{proof}
If $M$ is Gorenstein projective then $\FF M$ is an object of $\C(G)$ because $Y^M$ is an acyclic complex of projective $H$-modules (cf.\ \cite{Gil21}, Lemma 4.3 (4)). Conversely, if $V\in\C(G)$ then $V\cong B_0FY$ for some acyclic complex $Y$ of projective $H$-modules by Corollary \ref{SP_fib_cofib}. It then follows from Proposition \ref{GP_split_mono} (ii) and Remark \ref{complete_resolution} that $UV\cong UB_0FY\cong B_0Y$ is a Gorenstein projective $H$-module and that $U\FF$ is isomorphic to the identity functor on $\mathrm{GProj}(H)$. In particular, the functor $U:\C(G)\to\mathrm{GProj}(H)$ is essentially surjective and the functor $\FF:\mathrm{GProj}(H)\to \C(G)$ is faithful. However, $U$ on $\C(G)$ is faithful, too. This follows from the fact that the objects of $\C(G)$ are quotients of $I$-projective $G$-representations, hence are generated by their $I$-invariants (cf.\ Lemma \ref{summand_I_proj}).\\

It then follows from $U\FF\cong\id$ that $U$ is fully faithful on the essential image of $\FF:\mathrm{GProj}(H)\to\C(G)$. However, this functor is essentially surjective. Indeed, if $V\in\C(G)$ we may assume $V=B_0FY$ for some acyclic complex $Y$ of projective $H$-modules. Set $M=B_0Y$. By Remark \ref{complete_resolution} and \cite{Gil21}, Theorem 4.1, there is a map $f:Y\to Y^M$ of complexes such that $B_0f$ is the identity on $M$. We claim that the map $B_0Ff:V=B_0FY\to B_0FY^M=\FF M$ is an isomorphism in $\Rep^\infty_k(G)$. Since $U$ reflects monomorphisms and since $\FF M$ is generated by its $I$-invariants it suffices to see that $UB_0Ff$ is an isomorphism in $\Mod(H)$. This follows from Proposition \ref{GP_split_mono} (ii) because under the isomorphisms $UB_0FY\cong B_0Y=M\cong U\FF M$ the map $UB_0Ff$ is the identity on $M$.
\end{proof}

We have the following alternative characterization of the objects of $\C(G)$.

\begin{prop}\label{characterization_C(G)}
For any object $V\in\Rep^\infty_k(G)$ the following are equivalent.
\begin{itemize}
\item[(i)] We have $V\in\C(G)$.
\item[(ii)] The representation $V$ is generated by its $I$-invariants and there is a map $i:V\to W$ in $\Rep^\infty_k(G)$ such that $W$ is $I$-projective and $Ui$ is a cofibration in $\Mod(H)$.
\end{itemize}
In particular, any object $V$ of $\C(G)$ is both a quotient and a subobject of an $I$-projective $G$-representation and $UV$ is a Gorenstein projective $H$-module.
\end{prop}

\begin{proof}
That any object $V\in\C(G)$ is both a quotient and a subobject of an $I$-projective $G$-representation is true by definition. Moreover, $UV$ is Gorenstein projective by the proof of Theorem \ref{Cabanes}. Therefore, we only need to show that (i) and (ii) are equivalent. Assume that $V$ satisfies (i). By definition $V\cong B_0FY$ for some acyclic complex of projectives $Y\in\Ch(H)$. This gives an embedding $i:V\to FY_0$ where $FY_0$ is $I$-projective. Moreover, $V$ is generated by its $I$-invariants as seen in the proof of Theorem \ref{Cabanes}. Under the isomorphism $UV\cong UB_0FY\cong B_0Y$ in Proposition \ref{GP_split_mono} (ii) the $H$-module $\coker(Ui)$ gets identified with $B_1Y$. Thus, $\coker(Ui)$ is Gorenstein projective and $Ui$ is a cofibration. This shows (ii). Conversely, assume that $V$ satisfies (ii). Since $Ui$ is a cofibration we have $\coker(Ui)\in\mathrm{GProj}(H)$. By the functorial factorisations in $\Mod(H)$ we can embed $\coker(Ui)$ into a projective $H$-module via a cofibration. This way we inductively construct an acyclic complex
\[
0\longrightarrow UV\stackrel{Ui}{\longrightarrow} UW=Y_0\stackrel{d_0}{\longrightarrow}  Y_{-1}\stackrel{d_{-1}}{\longrightarrow}  Y_{-2}\stackrel{d_{-2}}{\longrightarrow} \ldots
\]
where $Y_i$ is a projective $H$-module for all $i\leq 0$. If we set $X_i=FY_i$ then Lemma \ref{unit_proj} and Lemma \ref{summand_I_proj} allow us to identify this complex with the $I$-invariants of the complex
\[
0\longrightarrow V\stackrel{i}{\longrightarrow} W\cong FUW=X_0\stackrel{Fd_0}{\longrightarrow}  X_{-1}\stackrel{Fd_{-1}}{\longrightarrow}  X_{-2}\stackrel{Fd_{-2}}{\longrightarrow} \ldots
\]
which is therefore $I$-exact. Composing it with an $I$-resolution of $V$ gives an $I$-exact complex $X$ of $I$-projective $G$-representations with $V=B_0X$. This shows $V\in\C(G)$.
\end{proof}

\begin{rem}\label{Cabanes_finite_reductive}
Assume that $H$ comes from a finite group $G$ with a split $BN$-pair of characteristic $p$ as in Example \ref{H_Gorenstein} (i). Then $H$ is selfinjective and $\mathrm{GProj}(H)=\Mod(H)$. Moreover, any monomorphism in $\Mod(H)$ is a cofibration. It follows from Proposition \ref{characterization_C(G)} that $\C(G)$ is the full subcategory of representations which are both a quotient and a subobject of an $I$-projective $G$-representation. In this case the finite dimensional objects of $\C(G)$ were first studied by Cabanes (cf.\ \cite{Cab}, \S1). Moreover, Theorem \ref{Cabanes} recovers \cite{Cab}, Theorem 2. Our arguments differ only gradually. Everything relies on the fact that the adjunction $(F,U)$ induces an equivalence between projective $H$-modules and $I$-projective $G$-representations (cf.\ Corollary \ref{proj_equiv}).
\end{rem}

Proposition \ref{characterization_C(G)} yields the following class of cofibrations in $\Rep^\infty_k(G)$.

\begin{lem}\label{CG_cofibration} If $V\in\Rep^\infty_k(G)$ is generated by its $I$-invariants then any map $i:V\to W$ as in Proposition \ref{characterization_C(G)} (ii) is a cofibration.
\end{lem}

\begin{proof}
By assumption the map $Ui:UV\to UW$ is a cofibration in $\Mod(H)$. Since $F$ is left Quillen the map $FUi:FUV\to FUW$ is a cofibration in $\Rep^\infty_k(G)$. We have a commutative square
\[
\begin{tikzcd}
FUV\arrow{r}{FUi}\arrow[swap]{d}{\varepsilon_V} & FUW\arrow{d}{\varepsilon_W}\\
V\arrow[swap]{r}{i} & W
\end{tikzcd}
\]
where $\varepsilon_V$ is surjective since $V$ is generated by its $I$-invariants. Moreover, $\varepsilon_W$ is an isomorphism by Lemma \ref{summand_I_proj}. Therefore, the composition $\varepsilon_W^{-1} i$ induces an isomorphism $f:V\to \im(FUi)$. Since the inclusion $g:\im(FUi)\to FUW$ is a cofibration by Lemma \ref{cofib_I_inv} so is $i=\varepsilon_W g f$.
\end{proof}

The difference between the functors $F$  and $\FF$ can be made explicit as follows.

\begin{lem}\label{relation_F_FF}
There is a pointwise exact sequence
\[
0\longrightarrow H_1FY^{(\cdot)}\longrightarrow F\longrightarrow\FF\longrightarrow 0
\]
of functors from $\Mod(H)$ to $\Rep^\infty_k(G)$. The functor $\FF:\Mod(H)\to\Rep^\infty_k(G)$ preserves surjections and its restriction $\FF:\mathrm{GProj}(H)\to\C(G)$ preserves mono\-morphisms.
\end{lem}

\begin{proof}
The first statement follows from Proposition \ref{GP_split_mono} (i). Since $F$ preserves surjections it follows that so does $\FF$. The final statement follows from Theorem \ref{Cabanes}. 
\end{proof}

For a suitable exact structure on $\C(G)$ the functor $\FF$ actually becomes exact.

\begin{cor}\label{C(G)_Frobenius}
With respect to the short $I$-exact sequences the category $\C(G)$ is a weakly idempotent complete Frobenius category and the functors $\FF:\mathrm{GProj}(H)\rightleftarrows\C(G):U$ are exact inverse equivalences.
\end{cor}

\begin{proof}
Recall from Lemma \ref{pullback} that a sequence in $\Rep^\infty_k(G)$ is $I$-exact if and only if it is a complex such that the sequence obtained by applying $U$ is exact in $\Mod(H)$. Therefore, everything follows from Theorem \ref{Cabanes}.
\end{proof}

\begin{rem}\label{warning_exactness}
As recalled earlier, in an exact category $\A$ an admissible exact sequence $X\stackrel{i}{\to}Y\stackrel{p}{\to}Z$ really is exact in the categorical sense, i.e.\ we have $i=\ker(p)$ and $p=\coker(i)$. If $0\to X\stackrel{i}{\to}Y\stackrel{p}{\to}Z\to 0$ is a short $I$-exact sequence in $\C(G)$ then the relations $i=\ker(p)$ and $p=\coker(i)$ hold in $\C(G)$ but not necessarily in the larger category $\Rep^\infty_k(G)$. In fact, in $\Rep^\infty_k(G)$ the sequence is a complex, $i$ is injective and $p$ is surjective. However, in $\Rep_k^\infty(G)$ the inclusion $\im(i)\subseteq\ker(p)$ might not be an isomorphism. In fact, $\im(i)$ is the subrepresentation of $\ker(p)$ generated by $U\ker(p)$. 
\end{rem}

Endowing $\C(G)$ with the structure of a Frobenius category as in Corollary \ref{C(G)_Frobenius} the functors $\FF$ and $U$ can be viewed as inverse equivalences of exact model categories $\mathrm{GProj}(H)\cong\C(G)$ and $\Ch(\mathrm{GProj}(H))\cong\Ch(\C(G))$. Therefore, a morphism $f:V\to W$ in $\C(G)$ is a weak equivalence if and only if $Uf$ is a weak equivalence in $\mathrm{GProj}(H)$ and hence in $\Mod(H)$. Consequently, the inclusion functor $i:\C(G)\to\Rep^\infty_k(G)$ preserves weak equivalences and $i$ admits a well-defined homotopy functor $\Ho(i):\Ho(\C(G))\to\Ho(\Rep^\infty_k(G))$. Denote by $Q:\Mod(H)\to\mathrm{GProj}(H)$ the cofibrant replacement functor. 

\begin{thm}\label{derived_right_inverse}
The composed functor
\[
\Ho(\C(G))\stackrel{\Ho(i)}{\longrightarrow}\Ho(\Rep^\infty_k(G))\stackrel{RU}{\longrightarrow}\Ho(\Mod(H))
\]
induced by the fully faithful functor $\C(G)\stackrel{i}{\to}\Rep_k^\infty(G)\stackrel{U}{\to}\Mod(H)$ is an equivalence of categories with inverse $\Ho(\FF)\Ho(Q)$. 
\end{thm}

\begin{proof}
Write $U^{(1)}$ for the functor $U:\Rep^\infty_k(G)\to\Mod(H)$ and $U^{(2)}$ for the induced functor $\C(G)\to\mathrm{GProj}(H)$. Since all objects of $\Rep^\infty_k(G)$ are fibrant the proof of \cite{HovBook}, Proposition 1.2.3, shows that the homotopy functor of the fibrant replacement functor of $\Rep^\infty_k(G)$ is isomorphic to the identity functor. Therefore, $RU^{(1)}\cong\Ho(U^{(1)})$.\\

If $j:\mathrm{GProj}(H)\to\Mod(H)$ denotes the inclusion then we have $U^{(1)}i=jU^{(2)}$ and hence $\Ho(U^{(1)})\Ho(i)=\Ho(j)\Ho(U^{(2)})$. However, the proof of \cite{HovBook}, Proposition 1.2.3, shows that $\Ho(j)$ is an equivalence of categories with inverse $\Ho(Q)$. Moreover, $\Ho(U^{(2)})$ is an equivalence of categories with inverse $\Ho(\FF)$ by Corollary \ref{C(G)_Frobenius}.
\end{proof}

In particular, Theorem \ref{derived_right_inverse} implies that the functor $RU$ (resp.\ $\Ho(i)$) admits a right (resp.\ left) inverse and that $\Ho(i)$ allows us to view $\Ho(\C(G))$ as a (not necessarily full) subcategory of $\Ho(\Rep^\infty_k(G))$. The restriction of $RU$ to this subcategory is an equivalence of categories
\[
\Ho(\C(G))\cong\Ho(\Mod(H)).
\]

\begin{rem}\label{trivial}
Due to the generality of our setup there are situations in which our results are not particularly helpful. If $G$ is discrete, for example, and if $I=1$ then $\Mod(H)=\Rep^\infty_k(G)$ and we simply study the identity functor $F=U=\id$. If $G$ is a pro-$p$ group and if $I=G$ then $H=k$ is selfinjective and $\C(G)$ is the category of trivial $G$-representations. In this case the equivalence $U:\C(G)\to\mathrm{GProj}(H)=\Mod(k)$ is given by the forgetful functor. In the situations of Example \ref{H_Gorenstein}, however, the category $\Mod(H)$ is comparatively easy to understand and yet strongly linked to $\Rep^\infty_k(G)$. Passing to $\Ho(\Mod(H))\cong\mathrm{GProj}(H)/\mathrm{Proj}(H)$ simplifies the situation further by getting rid of the objects of finite projective dimension. 
\end{rem}

Finally, we explain in which way the objects of $\C(G)$ appear in the theory of equivariant coefficient systems on Bruhat-Tits buildings as studied in \cite{Koh} and \cite{OS14}. To this aim we assume that $H$ is associated to $G=\mathbb{G}(K)$ and $I=I_C$ as in Example \ref{H_Gorenstein} (iv). Given $M\in\Mod(H)$ Ollivier and Schneider construct in \cite{OS14}, \S6.4, a functorial exact sequence of $H$-modules
\[
0\to\GP(M)_d\to\ldots\to\GP(M)_0\to M \to 0 
\]
where $d$ is the semisimple rank of $\mathbb{G}$. If $\mathbb{G}$ is semisimple then all $\GP(M)_i$ are Gorenstein projective by \cite{OS14}, Lemma 6.4. On the other hand, the second author functorially associates with $M\in\Mod(H)$ a $G$-equivariant coefficient system $\F(M)$ on the semisimple Bruhat-Tits building $\mathcal{X}$ of $G$ (cf.\ \cite{Koh}, \S3.2). The corresponding complex $\C^\mathrm{or}_c(\mathcal{X}_{(\bullet)},\F(M))$ of {\it oriented chains} is $I$-exact (cf.\ \cite{Koh}, Proposition 2.9). In fact, it is concentrated in degrees $0\leq i\leq d$ and admits a functorial isomorphism
\begin{equation}\label{iso_GP_chains}
\C^\mathrm{or}_c(\mathcal{X}_{(\bullet)},\F(M))^I\cong\GP(M)_\bullet
\end{equation}
in $\Ch(H)$ (cf.\ \cite{Koh}, Remark 3.24). In particular, there is a functorial isomorphism $H_0(\C^\mathrm{or}_c(\mathcal{X}_{(\bullet)},\F(M))^I)\cong M$ of $H$-modules (cf.\ \cite{Koh}, Theorem 3.21).

\begin{prop}\label{oriented_chains}
Assume that $\mathbb{G}$ is semisimple. For any $M\in\Mod(H)$ there is an isomorphism of complexes of smooth $k$-linear $G$-representations
\[
\C^\mathrm{or}_c(\mathcal{X}_{(\bullet)},\F(M))\cong\FF\GP(M)_\bullet.
\]
\end{prop}

\begin{proof}
This will follow from Theorem \ref{Cabanes} and (\ref{iso_GP_chains}) if we can show that the complex $\C^\mathrm{or}_c(\mathcal{X}_{(\bullet)},\F(M))$ consists of objects of $\C(G)$. To see this let $\sigma$ be a facet of $\mathcal{X}$ contained in the closure of $C$ and denote by $P_\sigma^\dagger$ the stabilizer of $\sigma$ in $G$. Setting $\XX_\sigma=\ind_I^{P_\sigma^\dagger}(k)$ we have the Gorenstein ring $H_\sigma^\dagger=\End_{P_\sigma^\dagger}(\XX_\sigma)^\op$ associated to $P_\sigma^\dagger$ and $I=I_C$ as in Example \ref{H_Gorenstein} (iii). Note that we may view $H_\sigma^\dagger$ as a subalgebra of $H$. Denote by $U_\sigma$, $F_\sigma$ and $\FF_\sigma$ our usual functors corresponding to $P_\sigma^\dagger$ and $H_\sigma^\dagger$.\\

In \cite{Koh}, Theorem 3.12, the second author constructs a fully faithful functor $\mathfrak{t}_\sigma:\Mod(H_\sigma^\dagger)\to\Rep^\infty_k(P_\sigma^\dagger)$ which is right inverse to $U_\sigma$. Given $N\in\Mod(H_\sigma^\dagger)$ the canonical complete resolution $Y^N_\bullet$ provides an embedding $N\hookrightarrow Y_0^N$ where $Y_0^N$ is an $H_\sigma^\dagger$-module of finite projective dimension. By our semisimplicity assumption the ring $H_\sigma^\dagger$ is selfinjective (cf.\ \cite{OS14}, Proposition 5.5). It follows that $Y_0^N$ is a projective $H_\sigma^\dagger$-module. The arguments given in the proof of \cite{Koh}, Theorem 3.12, then show that there is a functorial $H_\sigma^\dagger$-linear isomorphism $\mathfrak{t}_\sigma N\cong\FF_\sigma N$. Thus, $\mathfrak{t}_\sigma\cong\FF_\sigma$ as functors.\\

Consider the automorphism $j_\sigma$ of the $k$-algebra $H_\sigma^\dagger$ introduced in \cite{OS14}, \S3.3.1. Via scalar restriction it induces an automorphism of the category $\Mod(H_\sigma^\dagger)$ denoted by $N\mapsto N(\epsilon_\sigma)$. For $M\in\Mod(H)$ denote by $M_\sigma\in\Mod(H_\sigma^\dagger)$ the scalar restriction of $M$ along the inclusion $H_\sigma^\dagger\hookrightarrow H$. By construction and \cite{OS14}, Lemma 3.7, any term in the oriented chain complex of $\F(M)$ is a finite direct sum of representations of the form $V_\sigma=\ind_{P_\sigma^\dagger}^G(\FF_\sigma M_\sigma(\epsilon_\sigma))$. By exactness of compact induction the embedding $j: \FF_\sigma M_\sigma(\epsilon_\sigma)\to F_\sigma Y_0^{M_\sigma(\epsilon_\sigma)}$ induces an embedding $i=\ind_{P_\sigma^\dagger}^G(j):V_\sigma\to\ind_{P_\sigma^\dagger}^G(F_\sigma Y_0^{M_\sigma(\epsilon_\sigma)})$ of $G$-representations. The latter is $I$-projective because compact induction commutes with arbitrary direct sums and satisfies $\ind_{P_\sigma^\dagger}^G(\XX_\sigma)\cong \XX$. Moreover, the $P_\sigma^\dagger$-represen\-tation $\FF_\sigma M_\sigma(\epsilon_\sigma)$ is generated by its $I$-invariants (cf.\ Proposition \ref{characterization_C(G)}) and the $G$-representation $V_\sigma$ is generated by its $P_\sigma^\dagger$-subrepresentation $\FF_\sigma M_\sigma(\epsilon_\sigma)$. It follows that the $G$-representation $V_\sigma$ is generated by its $I$-invariants. By Proposition \ref{characterization_C(G)} (ii) we are left to show that $Ui$ is a cofibration.\\

Since $\FF_\sigma M_\sigma(\epsilon_\sigma)$ embeds into an $I$-free $P_\sigma^\dagger$-representation the pro-$p$ radical $I_\sigma$ of the parahoric subgroup $P_\sigma$ acts trivially on $\FF_\sigma M_\sigma(\epsilon_\sigma)$. By \cite{Koh}, Proposition 4.17, there is a natural $H$-linear isomorphism $H\otimes_{H_\sigma^\dagger}U_\sigma\FF_\sigma M_\sigma(\epsilon_\sigma)\cong UV_\sigma$ inducing an isomorphism $\coker(Ui)\cong H\otimes_{H_\sigma^\dagger}\coker(U_\sigma j)$. Since $H_\sigma^\dagger$ is selfinjective the $H_\sigma^\dagger$-module $\coker(U_\sigma j)$ is Gorenstein projective. As in \cite{OS14}, Lemma 6.4, it follows that $\coker(Ui)$ is Gorenstein projective over $H$.
\end{proof}

\begin{rem}\label{FF_compact_ind}
Assume that $\mathbb{G}$ is semisimple and take up the notation from the proof of Proposition \ref{oriented_chains}. Let $N$ be any $H_\sigma^\dagger$-module and let $M=H\otimes_{H_\sigma^\dagger}N$. Then $N$ is Gorenstein projective over $H_\sigma^\dagger$ and $M$ is Gorenstein projective over $H$. We have seen that $\ind_{P_\sigma^\dagger}^G(\FF_\sigma N)$ is an object of $\C(G)$. Moreover, $\FF M$ is an object of $\C(G)$ by Theorem \ref{Cabanes}. By Theorem \ref{Cabanes} and the proof of Proposition \ref{oriented_chains} there are canonical $H$-linear isomorphisms
\[
U\FF (H\otimes_{H_\sigma^\dagger}N)\cong H\otimes_{H_\sigma^\dagger}N\cong H\otimes_{H_\sigma^\dagger}U_\sigma\FF_\sigma N\cong U\ind_{P_\sigma^\dagger}(\FF_\sigma N).
\]
It follows from Theorem \ref{Cabanes} that this is induced by an isomorphism of functors $\FF\circ(H\otimes_{H_\sigma^\dagger}(\cdot))\cong\ind_{P_\sigma^\dagger}^G\circ\FF_\sigma$ from $\Mod(H_\sigma^\dagger)$ to $\C(G)$.
\end{rem}

On the other hand, any representation $V\in\Rep^\infty_k(G)$ gives rise to a $G$-equivariant coefficient system $\underline{\underline{V}}$ as in \cite{OS14}, \S3.1. The corresponding oriented chain complex $\C_c^\mathrm{or}(\mathcal{X}_{(\bullet)},\underline{\underline{V}})$ will generally not consist of objects of $\C(G)$ as the following example shows.

\begin{ex}\label{not_C(G)}
Let $\sigma$ be a facet of $\mathcal{X}$ contained in the closure of $C$. Moreover, denote by $I_\sigma$ the pro-$p$ radical of the parahoric subgroup $P_\sigma$. Then $W=\ind_{P_\sigma^\dagger}^G(V^{I_\sigma})$ is a direct summand of one of the terms of the oriented chain complex of $\underline{\underline{V}}$. However, if the $P_\sigma^\dagger$-representation $V^{I_\sigma}$ is not generated by its $I$-invariants then neither is the $G$-representation $W$. Indeed, if $W'=\ind_{P_\sigma^\dagger}^G(k[P_\sigma^\dagger]\cdot V^I)$ then $W'$ is a proper subrepresentation of $W$. On the other hand, the inclusion $W'\hookrightarrow W$ induces an isomorphism on $I$-invariants as follows from \cite{Koh}, Proposition 4.17. Therefore,
\[
k[G]\cdot W^I=k[G]\cdot (W')^I= W'\subsetneqq W
\]
and $W$ does not lie in $\C(G)$ by Proposition \ref{characterization_C(G)}. For a concrete example consider the $\mathrm{GL}_2(\mathbb{Q}_p)$-representation $V$ described in \cite{OS14}, Remark 3.2 (3).
\end{ex}


Competing interests: The authors declare none.

\vspace{1cm}


Universit\"at Duisburg-Essen\\
Fakult\"at f\"ur Mathematik\\
Thea-Leymann-Stra{\ss}e 9\\
D--45127 Essen, Germany\\
{\it E-mail addresses:} {\ttfamily nicolas.dupre@uni-due.de}, {\ttfamily jan.kohlhaase@uni-due.de}


\end{document}